\documentclass{amsproc}

\theoremstyle{definition}

\usepackage{geometry}
\theoremstyle{remark}

\numberwithin{equation}{section}

\usepackage{graphicx}
\usepackage{latexsym}
\usepackage{amsmath}
\usepackage{amssymb}
\usepackage{amsfonts}
\usepackage{verbatim}
\usepackage{mathrsfs}
\usepackage[latin1]{inputenc}
\usepackage{color}
\usepackage{caption} 

\newtheorem{tm}{Theorem}[section]
\newtheorem{rk}{Remark}[section]
\newtheorem{ap}{Assumption}[section]

\newtheorem{prop}{Proposition}[section]
\newtheorem{lm}{Lemma}[section]
\newtheorem{cor}{Corollary}[section]

\newcommand{\E}{\mathbb E}

\newcommand{\N}{\mathbb N}

\newcommand{\bi}{\mathbf i}
\newcommand{\bs}{\mathbf s}

\newcommand{\HH}{\mathbb H}

\newcommand{\<}{\langle}
\renewcommand{\>}{\rangle}
\allowdisplaybreaks[4]

\begin{document}

\title[Regularized splitting methods for SlogS equation]
{Structure-preserving splitting methods for stochastic  logarithmic Schr\"odinger equation  via regularized energy approximation}

\author{Jianbo Cui}
\address{Department of Applied Mathematics, The Hong Kong Polytechnic University, Hung Hom, Hong Kong}
\curraddr{}
\email{jianbo.cui@polyu.edu.hk}
\thanks{}
\author{Jialin Hong}
\address{1. LSEC, ICMSEC, 
			Academy of Mathematics and Systems Science, Chinese Academy of Sciences, Beijing,  100190, China\qquad
			2. School of Mathematical Science, University of Chinese Academy of Sciences, Beijing, 100049, China}
\curraddr{}
\email{hjl@lsec.cc.ac.cn}
\thanks{
}

\author{Liying Sun}
\address{LSEC, ICMSEC, 
			Academy of Mathematics and Systems Science, Chinese Academy of Sciences, Beijing,  100190, China
		}
\curraddr{}
\email{liyingsun@lsec.cc.ac.cn}
\thanks{This work is supported by National Natural Science Foundation of China (No.
91630312, No. 91530118, No.11021101 and No. 11290142) and by the China Postdoctoral
Science Foundation (No. 2021M690163, No. BX2021345). The research of J. C. is partially supported by start-up funds from Hong Kong Polytechnic University and the CAS AMSS-PolyU Joint Laboratory of Applied Mathematics. 
}


\subjclass[2010]{Primary 60H35; Secondary  35Q55, 35R60, 65M12}

\keywords{stochastic  Schr\"odinger equation,
logarithmic nonlinearity,
energy regularized approximation, 
structure-preserving splitting method,
strong and weak convergence.
}

\date{\today}

\dedicatory{}

\begin{abstract}
In this paper, we study two kinds of structure-preserving splitting methods, including the Lie--Trotter type splitting method and the finite difference type method, for the  stochastic logarithmic Schr\"odinger equation (SlogS equation) via a regularized energy approximation. 
We first introduce a regularized SlogS equation with a small parameter $0<\epsilon\ll1$ which approximates the SlogS equation and avoids the singularity near zero density. 
Then we present a priori estimates, the regularized entropy and energy, and 
the stochastic symplectic structure of the proposed numerical methods. 
Furthermore, we derive both the strong convergence rates and the convergence rates of the regularized entropy and energy. To the best of our knowledge, this is the first result concerning the construction and analysis of numerical methods for stochastic Schr\"odinger equations with logarithmic nonlinearities. 
\end{abstract}

\maketitle


\section{Introduction}

In this paper, we focus on the SlogS equation
\begin{align}
\label{SlogS}
d u(t)&=\mathbf i \Delta u(t)dt+\mathbf i \lambda u(t)\log(|u(t)|^2)dt+ \widetilde g(u(t))\star dW(t), \; t>0,\\\nonumber 
u(0)&=u_0, 
\end{align}
where $\lambda\in \mathbb R/\{0\}$ measures the force of the logarithmic nonlinearity, $\Delta$ is the Laplacian operator on $\mathcal O\subset \mathbb R^d$ with $\mathcal O$ being either $\mathbb R^d$ or a bounded domain with homogeneous Dirichlet or periodic boundary condition, and $d\in \mathbb N^+$ is the spatial dimension. Here $W(t)$ is a $Q$-Wiener process, i.e., $W(t)=\sum_{k\in \mathbb N^+}Q^{\frac 12}e_i\beta_k(t)$ with $\{\beta_k\}_{k\in\mathbb N^+}$ being a sequence of independent Brownian motions on a probability space $(\Omega,$ 
$\mathcal F, (\mathcal F_t)_{t\ge 0},\mathbb P).$ 
The operator $Q^{\frac 12}$ {is} bounded on $\mathbb H:=L^2(\mathcal O;\mathbb C)$ satisfying $\sum_{i\in \mathbb N^+}\|Q^{\frac 12}e_{i}\|^2<\infty$, and $\{e_i\}_{i\in\mathbb N^+}$ is an orthonormal basis of $\mathbb H$. 
Here $\widetilde g$ is a continuous function and 
\begin{align*}
\widetilde g(u)\star dW(t) &:=-\frac 12\sum_{i\in \mathbb N^+}|Q^{\frac 12}e_i|^2\Big(|g(|u|^2)|^2u\Big)dt\\\nonumber 
&\quad-\bi \sum_{i\in \mathbb N^+}g(|u|^2)g'(|u|^2) |u|^2 u
Im(Q^{\frac 12}e_i)Q^{\frac 12}e_i dt+ \mathbf i g(|u|^2)u dW(t)
\end{align*}
if $\widetilde g(x)=\bi g(|x|^2)x$ (multiplicative case), and  
$$\widetilde g(u)\star dW(t):= dW(t)$$ if $\widetilde g=1$ (additive case). 
The SlogS equation has wide applications in quantum mechanics, quantum optics, nuclear physics, transport and diffusion phenomena, open quantum system, Bose-Einstein condensations, etc. (see e.g. \cite{AZ03,BM76,Hef85,MFGL03,Yas78,Zlo10}). 

The logarithmic nonlinearity possesses several features which make the logarithmic Schr\"odinger equation unique among nonlinear wave equations.
For instance, the logarithmic nonlinearity is not locally Lipschitz and causes singularity near vacuum. 
The large time behavior, like the dispersive phenomenon depending on the sign of $\lambda$, is  totally different from that in the Schr\"odinger equation with smooth nonlinearity (see e.g. \cite{CG18,Caz83}). Moreover, the randomness of the driving noise destroys many well-known conservation laws and structures. 
Several physical quantities, such as the mass, momentum and energy, may be no longer conservative for the SlogS equation (see e.g. \cite{BRZ17,C2020}). 
Only when $\widetilde g(x)=\bi g(|x|^2)x$ with a continuous real-valued  function $g$ and $W(t)$ is $L^2(\mathcal O;\mathbb R)$-valued,  the mass conservation law holds. 
The well-posedness of the SlogS equation driven by the linear multiplicative noise ($\widetilde g(x)=\bi x$) has been shown in \cite{BRZ17}.
{Recently, the author in \cite{C2020} obtains the well-posedness of the SlogS equation with general diffusion coefficients (including $\widetilde g(x)=\bi g(|x|^2)x$ and $\widetilde g(x)=1$).}  

Despite various and fruitful numerical results of stochastic Schr\"odinger equations with smooth nonlinearities {(see e.g. \cite{AC18,BC2020a,CH16,CH17,CHL16b,CHLZ17,BD06,HW19} and references therein)},  the numerical analysis of stochastic Schr\"odinger equations  
with non-locally Lipschitz nonlinearities, especially for the SlogS equations, is far from being well understood and confronts several challenges.  One is that the direct numerical discretization often produces the numerical vacuum which are difficult to deal with when computing the logarithmic nonlinearity.   
Another challenge lies on the mutual effect of the random noise and the non-locally Lipschitz coefficient, which leads that the existing numerical approach for analyzing the geometric structures and convergence analysis of numerical methods are not available for SlogS equations.  To overcome these issues, we will  introduce a regularized problem of \eqref{SlogS} which is used to show the well-posedness of the SlogS equation in \cite{C2020}.
Then we show that the regularized energy of the regularized SlogS equation is well-defined, and thus the regularized SlogS equation is a stochastic Hamiltonian partial differential equation whose phase flow preserves the stochastic symplectic structure. 
The a priori estimates and convergence results of the regularized SlogS equation to \eqref{SlogS} are also presented.

Furthermore, we propose structure-preserving splitting methods of different types, including the Lie--Trotter splitting methods and finite difference methods, based on the regularized SlogS equation, to inherit the intrinsic properties of original systems.
We study several important features of the proposed numerical methods, including the moment estimates of regularized entropy and energy, the mass evolution law and the symplectic structure. 
Based on the structure-preserving properties of the regularized SlogS equation and numerical methods, error estimates in both strong and weak convergence senses are established between the solutions of \eqref{SlogS} and the regularized splitting methods.  To the best of our knowledge, this is the first result on the construction of numerical methods   and their numerical analysis, including structure-preserving properties, strong and weak  convergence, for the SlogS equation. 
We would like to mention that  the proposed numerical methods are even new in the deterministic case, i.e., $\widetilde g=0$, and 
all the numerical results in this paper still hold in the deterministic case.


The reminder of this article is organized as follows. In section 2, we introduce some basic notations and present  the well-posedness result of the SlogS equation and its regularized version. 
Section 3 is devoted to constructing and analyzing the Lie--Trotter type splitting method, including $\epsilon$-independent a priori estimate, mass evolution law and symplectic structure. 
In section 4, we propose the finite difference type splitting method, and prove its convergence of the energy functional and strong convergence. Throughout this article,  $C$ denotes
various positive constants which may change from line to line.

\section{Regularized SlogS equation}

In this section, we introduce some necessary notations, the well-posedness result of  SLogS equation \eqref{SlogS}, as well as the properties of the regularized SLogS equation. 

\subsection{Preliminary}
Denote $\mathbb H$ with the product $\<u,v\>:= Re [\int_{\mathcal O} u\bar vdx],$ $u,v\in \mathbb H.$ Let $W^{k,p}, k\in \mathbb N,p\in \mathbb N^+$ be the classic Sobolev spaces and $\mathbb H^k=W^{k,2}.$  Denote $L^p:=L^p(\mathcal O;\mathbb C),$ $p\ge 1.$ In order to bound the entropy $F(\rho):=\int_{\mathcal O}(\rho\log \rho -\rho)dx, \rho=|u|^2$ in the case of $\mathcal O=\mathbb R^d$, we introduce the weighted square integrable space 
$$L^2_{\alpha}:=\{v\in {\mathbb H} | \; x \mapsto (1+|x|^2)^{\frac \alpha 2}v(x)\in {\mathbb H}\}$$
with the norm $\|v\|_{L_\alpha^2(\mathbb R^d)}:=\|(1+|\cdot|^2)^{\frac \alpha 2}v(\cdot)\|_{L^2(\mathbb R^d)}, \alpha\ge 0.$ 
For convenience, we always assume that $u_0\in \mathbb H^1\cap L^2_{\alpha}, $ $\alpha\in(0,1],$ is $\mathcal F_0$-measurable and has any finite $p$-moment, $p\in \mathbb N^+$. The main assumption on $W$ and $\widetilde g$ is stated as follows.
\begin{ap}\label{main-as}
	The Wiener process $W$ and $\widetilde g$  satisfy one of the following conditions:
	\begin{enumerate}
		\item $\{W(t)\}_{t\ge 0}$ is $\mathbb H$-valued, 
		$\widetilde g=1$ or $\widetilde g(x)=\mathbf ig(|x|^2)x$ {with} $g \in \mathcal C^2_b(\mathbb R)$ satisfying the growth condition,
		\begin{align*}
		\sup_{x\in [0,\infty)}|g(x)|+\sup_{x\in [0,\infty)}|g'(x)x|+\sup_{x\in [0,\infty)}|g''(x)x^2|\le C_g,
		\end{align*} 
		the one-side Lipschitz continuity condition, i.e.,  for any $x,y\in \mathbb C,$
		\begin{align}\label{con-g1}
		|(\bar y-\bar x)(g'(|x|^2)g(|x|^2)|x|^2x-g'(|y|^2)g(|y|^2)|y|^2y)| \le C_g|x-y|^2,
		\end{align}
		and 
		\begin{align}\label{con-g}
		(x+y)(g(|x|^2)-g(|y|^2))\le C_g|x-y|, \quad x,y\in [0,\infty).
		\end{align}
		\item $ {\{W(t)\}_{t\ge 0}}$ is $L^2(\mathcal O;\mathbb R)$-valued and $\widetilde g(x)=\mathbf ig(|x|^2)x$ with $g \in \mathcal C^1_b(\mathbb R)$ satisfying \eqref{con-g} and the growth condition 
		\begin{align*}
		\sup_{x\in [0,\infty)}|g(x)|+\sup_{x\in [0,\infty)}|g'(x)x|\le C_g.
		\end{align*} 
	\end{enumerate}
\end{ap}

The functions like $a, \frac a{b+x},\frac {ax}{b+cx},  \frac {ax}{b+cx^2}$ with  $b,c>0,$ will satisfy the above conditions on $g$ in Assumption \ref{main-as}. 

\begin{tm}[see~\cite{C2020}]
	\label{mild-general}
	Let $T>0$ and Assumptions \ref{main-as} hold,  $u_0\in \mathbb H^1\cap L^2_{\alpha},$ $\alpha\in(0,1],$ be $\mathcal F_0$-measurable and have any finite $p$th moment with $p\geq1$.   Assume that $\sum_{i\in\mathbb N^+} \|Q^{\frac 12}e_i\|_{L^2_{\alpha}}^2+\|Q^{\frac 12}e_i\|_{\mathbb H^1}^2 <\infty$ when $\widetilde g=1$ and that  $\sum_{i\in\mathbb N^+} \|Q^{\frac 12}e_i\|_{\mathbb H^1}^2+\|Q^{\frac 12}e_i\|_{W^{1,\infty}}^2<\infty$ when $\widetilde g(x)= \mathbf  ig(|x|^2)x$.
	Then there exists a unique mild solution $u$ in $C([0,T];\mathbb H)$ for Eq. \eqref{SlogS}.
	Moreover, there exists $C(Q,T,\lambda,p,u_0,\alpha,\widetilde g)>0$ such that 
	\begin{align*}
	\E\Big[\sup_{t\in [0,T]}\|u(t)\|_{\mathbb H^1}^p\Big]+\E\Big[\sup_{t\in [0,T]}\|u(t)\|_{L^2_{\alpha}}^p\Big]\le C(Q,T,\lambda, p,u_0,\alpha,\widetilde g).
	\end{align*}
\end{tm}

\subsection{Regularized  energy of regularized SlogS equation}
To deal with the logarithmic nonlinearity, we introduce the regularized  SLogS equation with $0<\epsilon \ll 1,$  
\begin{align}\label{Reg-SlogS}
du^{\epsilon}=\mathbf i \Delta u^{\epsilon}dt+\mathbf i \lambda u^{\epsilon} f_{\epsilon}(|u^{\epsilon}|^2)dt+ \widetilde g(u^{\epsilon})\star dW(t), \;u^{\epsilon}(0)=u_0.
\end{align}	
The regularized energy is defined by $H_{\epsilon}(u):= \frac 12\|\nabla u\|^2-\frac {\lambda}2 F_{\epsilon}(|u|^2)$, where  $F_{\epsilon}(\rho)=\int_{\mathcal O}\int_0^{\rho}f_{\epsilon}(s)ds dx$ is the regularized entropy and $f_{\epsilon}(\cdot)$ is a suitable approximation of $\log(\cdot).$   Formally speaking, we expect that $H_{\epsilon}(u)$ approximates the  energy of Eq. \eqref{SlogS}  $H(u):=\frac 12\|\nabla u\|^2-\frac \lambda{2} F(|u|^2).$
To this end, we impose the following assumption on $f_{\epsilon}$.

\begin{ap}\label{main-reg-fun}
	The regularization function $f_{\epsilon}$ satisfies
	\begin{enumerate}
		\item[(A1)]\label{A1} $|f_{\epsilon}(|x|^2)|\le  C(1+|\log(\epsilon)|)$ \; {$\forall~x\in \mathbb C.$}
		\item[(A2)]\label{A2} $|Im[(f_{\epsilon}(|x_1|^2)x_1-f_{\epsilon}(|x_2|^2)x_2)(\bar x_1-\bar x_2)]|\le C|x_1-x_2|^2$ \; {$\forall~x_1,x_2 \in \mathbb C.$}
		\item[(A3)]\label{A3} $\lim\limits_{\epsilon\to0} f_{\epsilon}(|x|^2)=\log(|x|^2)$ \; {$\forall~x\in \mathbb C.$}
		\item[(A4)]\label{A4} $|f_{\epsilon}(|x|^2)-\log(|x|^2)|\le C(\epsilon+(\epsilon|x|^2)^{\delta})$, $\delta\in (0, 1]$ when $|x|\ge 1,$ 
		and $|f_{\epsilon}(|x|^2)-\log(|x|^2)|\le C(\frac {\epsilon}{\epsilon+|x|^2}+\epsilon)$  when $|x|\le 1.$
		\item[(A5)]\label{A5} $|\frac {\partial f_{\epsilon}(|x|^2)}{\partial |x|}|\le C\frac {|x|}{\epsilon+|x|^2}$\; {$\forall~x\in \mathbb C.$} 
	\end{enumerate}
\end{ap}
One typical example satisfying  Assumption \ref{main-reg-fun} is  $f_{\epsilon}(|x|^2)=\log(\frac {\epsilon+|x|^2}{1+\epsilon|x|^2}),$ whose corresponding regularized  energy is
$ H_{\epsilon}(u^{\epsilon})$ with $ F_{\epsilon}(|u^{\epsilon}|^2)=\int_{\mathcal O}\Big(|u^{\epsilon}|^2\log(\frac {|u^{\epsilon}|^2+\epsilon}{1+|u^{\epsilon}|^2\epsilon})+\epsilon \log(|u^{\epsilon}|^2+\epsilon)-\frac1{\epsilon}\log(\epsilon |u^{\epsilon}|^2+1)-\epsilon\log(\epsilon)\Big) dx.$ 
Assumption \ref{main-reg-fun} is crucial to obtain the strong convergence result and  the H\"older regularity estimate for  Eq. \eqref{Reg-SlogS} (see e.g. \cite{C2020}), which is illustrated in the following lemma. We  would like to remark that $f_{\epsilon}(x)=\log(\epsilon+x)$ or $\log(\epsilon+\sqrt{x})^2$  fails to satisfy $(A 1)$. 

\begin{lm}\label{lm-con}
	Let the condition of Theorem \ref{mild-general} and Assumption \ref{main-reg-fun} hold.   Let  $u^0:=u$ be the mild solution of Eq. \eqref{SlogS}. 
	Then there exists a unique mild solution $u^{\epsilon}$ of Eq. \eqref{Reg-SlogS}. 
	For $p\ge1$, there exist  $C'(Q,T,\lambda,p,u_0,\widetilde g)>0$ and $C'(Q,T,\lambda,p,u_0,\alpha,\widetilde g)>0$ such that for $\epsilon\in [0,1],$
	\begin{align*}
	\E\Big[\|u^{\epsilon}(t)-u^{\epsilon}(s)\|^p\Big] &\le C'(Q,T,\lambda,p,u_0,\widetilde g) |t-s|^{\frac p2}, \\
	\E\Big[\sup_{t\in [0,T]}\|u^{\epsilon}(t)\|_{\mathbb H^1}^p\Big]+ \E\Big[\sup_{t\in [0,T]}\|u^{\epsilon}(t)\|_{L_\alpha^2}^p\Big] &\le C'(Q,T,\lambda,p,u_0,\alpha,\widetilde g).
	\end{align*}
	Moreover, for $p\ge1$ and $\delta\in \big(0, \max(\frac 2{\max(d-2,0)},1)\big),$   there exist $C(Q,T,\lambda,p,u_0,\widetilde g)>0$ and   $C(Q,T,\lambda,p,u_0,\alpha,\delta ,\widetilde g)>0$  such that when $\mathcal O$ is a bounded domain,
	\begin{align*}
	\E\Big[\sup_{t\in [0,T]}\|u^0(t)-u^{\epsilon}(t)\|^p\Big] \le C(Q,T,\lambda,p,u_0,\delta, \widetilde g)(\epsilon^{\frac p2}+\epsilon^{\frac {{\delta}p} 2}),
	\end{align*}
	and when $\mathcal O=\mathbb R^d,$
	\begin{align*}
	\E\Big[\sup_{t\in [0,T]}\|u^0(t)-u^{\epsilon}(t)\|^{p}\Big]
	&\le C(Q,T,\lambda,p,u_0,\alpha,\delta ,\widetilde g)(\epsilon^{\frac {\alpha p}{2\alpha+d}}+\epsilon^{\frac {{\delta}p} 2}).
	\end{align*}	
\end{lm}

The strong convergence of $u^{\epsilon}$ in \eqref{Reg-SlogS} provides a systematic way to study the numerical schemes of the SLogS equation in sections 3 and 4. We present a sketch of the proof of Lemma \ref{lm-con} in Appendix.  
The evolutions of the mass  $M(u):=\|u\|^2$ and the weighted mass $M_{\alpha}(u):=\|u\|_{L^2_{\alpha}}^2$ of \eqref{Reg-SlogS} are presented in Proposition  \ref{prop-evo} in Appendix.
Moreover, when $\widetilde g=1$ or $\widetilde g(x)=\bi x$, it can be verified that there is no error between the masses $M(u^0)$ and $M(u^\epsilon).$ 
Below, we show that the regularized energy of \eqref{Reg-SlogS} is well-defined.

\begin{lm}\label{lm-mod-en}
	Let Assumption \ref{main-reg-fun} and the condition of Theorem \ref{mild-general} hold. Let $u^{0}$ and $u^\epsilon$ be the mild solutions of  Eq. \eqref{SlogS} and Eq. \eqref{Reg-SlogS}, respectively. 
	The regularized energy is well-defined and satisfies that for $p\ge1,$ 
	\begin{align*}
	\E\Big[\sup\limits_{t\in[0,T]} H_{\epsilon}^{p}(u^{\epsilon}(t))\Big]\le  C(Q,T,\lambda,p,u_0,\alpha,\widetilde g).
	\end{align*}
\end{lm}

\begin{proof}
	Denote $\rho^{\epsilon}=|u^{\epsilon}|^2$ and $\rho^0=|u^0|^2.$
	According to Theorem \ref{mild-general} and Lemma \ref{lm-con}, applying the weighted interpolation inequality 
	\begin{align}\label{wei-sob}
	\|v\|_{L^{2-2\eta}}\le C \|v\|_{L^2_{\alpha}}^{\frac {d\eta}{2\alpha(1-\eta)}}\|v\|^{1-\frac {d\eta}{2\alpha(1-\eta)}}
	\end{align} 
	for $\alpha>\frac {d\eta}{2-2\eta}, \alpha\in (0,1]$,
	we have that $${\mathbb E}\Big[\sup_{t\in [0,T]}H^p(u^{0}(t))\Big]+\E\Big[\sup_{t\in [0,T]}\|{u^\epsilon}(t)\|_{\mathbb H^1}^{p}\Big]\le  C(Q,T,\lambda,p,u_0,\widetilde g).$$
	It only suffices to show the boundedness of 
	$\E\Big[\sup\limits_{t\in[0,T]} F_{\epsilon}(|u^{\epsilon}(t)|^2)^{p}\Big]$. 
	By using  (A4) and the property of the logarithmic function, we have that for small $\eta\in[0,1],$
	\begin{align*}
	&\quad|F_{\epsilon}(\rho^{\epsilon})-F_0(\rho^{\epsilon})|\\
	&\le \int_{\mathcal O\cap\{\rho^{\epsilon}\ge 1\}}\int_0^{\rho^{\epsilon}}|f_{\epsilon}(s)-f_0(s)|dsdx +\int_{\mathcal O\cap\{\rho^{\epsilon}\le 1\}}\int_0^{\rho^{\epsilon}}|f_{\epsilon}(s)-f_0(s)|dsdx\\
	&\le \int_{\rho^{\epsilon}\ge 1} \int_0^{\rho^{\epsilon}}C(\epsilon+(\epsilon|s|)^{\delta})dsdx+\int_{\rho^{\epsilon}\le 1} \int_0^{\rho^{\epsilon}}C(\frac {\epsilon}{\epsilon+s}+\epsilon)dsdx\\
	&\le \int_{\mathcal O} \int_0^{\rho^{\epsilon}}C(\epsilon+(\epsilon|s|)^{\delta})dsdx+ C\int_{\rho^{\epsilon}\le 1} \epsilon(\log(\epsilon+|u^{\epsilon}|^2)-\log(\epsilon))dx+C\epsilon \|u^{\epsilon}\|^2\\
	&\le C\epsilon \|u^{\epsilon}\|^2+C\epsilon^{\delta} \|u^{\epsilon}\|_{L^{2+2\delta}}^{2+2\delta}+C\epsilon^{\eta} |\log(\epsilon)|^{\eta} |u^{\epsilon}\|_{L^{2-2\eta}}^{2-2\eta}.
	\end{align*}
	
	Using \eqref{wei-sob}
	with $\alpha>\frac {d\eta}{2-2\eta}, \alpha\in (0,1]$, the Gagliardo--Nirenberg interpolation inequality,
	\begin{align}\label{gn-sob}
	\|v\|_{L^{2+2\delta}}\le C \|\nabla v\|^{\frac {d\delta}{2+2\delta}}\|v\|^{1-\frac {d\delta}{2+2\delta }}, \delta\in \big(0,\frac 2{\max(d-2,0)}\big)
	\end{align}
	and the uniform boundedness of $u^{\epsilon}$ in $\mathbb H^1$, we achieve that  
	\begin{align*}
	\E\Big[\sup_{t\in[0,T]}|F_{\epsilon}(\rho^{\epsilon}(t))-F(\rho^{\epsilon}(t))|^p\Big]&\le C(Q,T,\lambda,p,u_0,\widetilde g).
	\end{align*}
	Therefore, it suffices to show the uniform boundedness of $F(|u^{\epsilon}|^2).$
	Adopting \eqref{gn-sob} and \eqref{wei-sob}, 
	we get 
	\begin{align*}
	F(\rho^{\epsilon})
	&\le \|u^{\epsilon}\|^2+\int_{|u^{\epsilon}|\ge1} |u^{\epsilon}|^{2+2\delta}dx+\int_{|u^{\epsilon}|<1}|u^{\epsilon}|^{2-2\eta}dx\\
	&\le \|u^{\epsilon}\|^2+C\|\nabla u^{\epsilon}\|^{{d\delta}}\|u^{\epsilon}\|^{2+2\delta-d\delta}+ C\|u^{\epsilon}\|_{L^2_{\alpha}}^{\frac {d\eta}{\alpha}}\|u^{\epsilon}\|^{2-2\eta-\frac {d\eta}{\alpha}}.
	\end{align*}
	Using the uniform boundedness of $u^{\epsilon}$ in $\mathbb H^1\cap L_{\alpha}^2$ and Young's inequality, we obtain that 
	\begin{align*}
	\E\Big[\sup_{t\in[0,T]}|F(\rho^{\epsilon}(t))|^p\Big]&\le C(Q,T,\lambda,p,u_0,\alpha,\widetilde g).
	\end{align*}
	We complete the proof by combining the above estimates.
\end{proof}

\begin{rk}
	The  conditions (A1) and (A2) in Assumption \ref{main-reg-fun} are not necessary when proving the boundedness of the regularized  energy for Eq. \eqref{Reg-SlogS}. However, it is crucial  to analyze the strong convergence of numerical method due to loss of regularity in both time and space of the solutions to both Eq. \eqref{SlogS} and Eq. \eqref{Reg-SlogS}.
\end{rk}

\begin{cor}\label{cor-entropy}
	Let Assumption \ref{main-reg-fun} and the condition of Theorem \ref{mild-general} hold. Let $u^{0}$ and $u^\epsilon$ be the mild solutions of { Eq.} \eqref{SlogS} and { Eq.} \eqref{Reg-SlogS}, respectively. 
	Then the regularized  entropy is  strongly convergent to the entropy of Eq.  \eqref{SlogS}. Furthermore, for $p\ge1$ and $\delta\in \Big(0,\max(\frac 2{\max(d-2,0)},1)\Big),$ there exist $C(Q,T,\lambda,p,u_0,\delta,\widetilde g)>0$ and $C(Q,T,\lambda,p,u_0,\alpha,\delta,\widetilde g)>0$  such that when $\mathcal O$ is a bounded domain,
	\begin{align*}
	\sup_{t\in[0,T]}\E\Big[|F_{\epsilon}(\rho^{\epsilon}(t))-F(\rho^{0}(t))|^p\Big] \le C(Q,T,\lambda,p,u_0,\delta ,\widetilde g)(\epsilon^{\frac p2}+\epsilon^{\frac {{\delta}p} 2})
	\end{align*}
	and when $\mathcal O=\mathbb R^d,$
	\begin{align*}
	\sup_{t\in[0,T]}\E\Big[|F_{\epsilon}(\rho^{\epsilon}(t))-F(\rho^{0}(t))|^p\Big]
	&\le C(Q,T,\lambda,p,u_0,\alpha,\delta ,\widetilde g)(\epsilon^{\frac {\alpha p}{2\alpha+d}}+\epsilon^{\frac {{ \delta}p} 2}).
	\end{align*}
	
\end{cor}

\begin{proof}
Similar arguments as in the proof of Lemma \ref{lm-mod-en} yield that 
	\begin{align*}
	&|F_{\epsilon}(\rho^{\epsilon})-F (\rho^{0})|\\
	&=|F (\rho^{\epsilon})-F (\rho^0)|+ |F_{\epsilon}(\rho^{\epsilon})-F(\rho^{\epsilon})|\\
	&\le |F(\rho^{\epsilon})-F(\rho^0)|+C\epsilon \|u^{\epsilon}\|^2+C\epsilon^{\delta} \|u^{\epsilon}\|_{L^{2+2\delta}}^{2+2\delta}+C\epsilon^{\eta}(|\log(\epsilon)|)^{\eta} |u^{\epsilon}\|_{L^{2-2\eta}}^{2-2\eta},
	\end{align*}
	where $\eta>0$ is small enough and $\delta(d-2)\le 2$. 
	Notice that
	\begin{align*}
	&F(\rho^{\epsilon})-F(\rho^0)\\
	=& \int_{|u^0|>|u^{\epsilon}|} (\log(|u^{\epsilon}|^2)-\log(|u^0|^2))|u^{\epsilon}|^2dx
	+\int_{|u^0|>|u^{\epsilon}|} \log(|u^0|^2)(|u^{\epsilon}|^2-|u^0|^2)dx\\
	&+\int_{|u^0|<|u^{\epsilon}|} (\log(|u^{\epsilon}|^2)-\log(|u^0|^2)) |u^0|^2 dx
	+\int_{|u^0|<|u^{\epsilon}|} \log(|u^{\epsilon}|^2)(|u^{\epsilon}|^2-|u^0|^2) dx\\
	&+\|u^0\|^2-\|u^{\epsilon}\|^2.
	\end{align*}
	{\small The property of the logarithmic function and H\"older's inequality yield that 
		\begin{align*}
		&|F(\rho^{\epsilon})-F(\rho^0)|\\
		\le& |\|u^{\epsilon}\|^2-\|u^0\|^2|+
		C\|u^0-u^{\epsilon}\|(\|u^0+u^{\epsilon}\|+\|u^0\|_{L^{2+2\delta}}^{1+\delta}+\|u^0\|_{L^{2-2\eta}}^{1-\eta}+\|u^{\epsilon}\|_{L^{2+2\delta}}^{1+\delta}+\|u^{\epsilon}\|_{L^{2-2\eta}}^{1-\eta}).
		\end{align*}
		B}y making use of \eqref{gn-sob} and \eqref{wei-sob}, together with a priori estimate of $u^{\epsilon}$ in $\mathbb H^1$ and Lemma \ref{lm-con}, we complete the proof.
\end{proof}

Since the regularized energy of Eq. \eqref{Reg-SlogS} is well-defined, we further show that Eq. \eqref{Reg-SlogS} is an infinite-dimensional stochastic Hamiltonian system whose phase flow preserves stochastic symplectic structure via a standard argument in \cite{CH16}. 
The key of the proof lies on the fact  that  \eqref{Reg-SlogS} can be rewritten into  an infinite-dimensional stochastic Hamiltonian system
\begin{align*}
dP^\epsilon=-\frac {\delta H_{\epsilon}}{\delta Q^\epsilon}dt-\frac {\delta H_{Sto}}{\delta Q^\epsilon}\circ dW(t),\quad
dQ^\epsilon= \frac {\delta H_{\epsilon}}{\delta P^\epsilon}dt+\frac {\delta H_{Sto}}{\delta P^\epsilon}\circ dW(t)
\end{align*}
with $P^\epsilon$ (resp., $Q^\epsilon$) being the real (resp., imaginary) part of the solution $u^{\epsilon}.$ Here $H_{Sto}:=\int_{\mathcal O} \int_{0}^{(P^\epsilon)^2+(Q^\epsilon)^2}g(s)ds dx$ when $\widetilde g(x)=\mathbf i g(|x|^2)x$.
For the additive noise case, i.e., $\widetilde g=1$, the proof is similar.  
\begin{prop}
	Let Assumption \ref{main-reg-fun} and the condition of Theorem \ref{mild-general} hold. 
	Assume in addition that $\{W(t)\}_{t\ge 0}$ is $\mathbb H$-valued, $\widetilde g=1$ or that  $\{W(t)\}_{t\ge 0}$ is $L^2(\mathcal O;\mathbb R)$-valued, $\widetilde g(x)=\mathbf ig(|x|^2)x.$
	The phase flow of Eq.  \eqref{Reg-SlogS} preserves the stochastic symplectic structure, i.e., 
	$$\bar \omega(t)=\int_{\mathcal O}{\rm d}P^\epsilon(t)\wedge {\rm d}Q^\epsilon(t) dx=\int_{\mathcal O}{\rm d}P^\epsilon(0)\wedge {\rm d}Q^\epsilon(0) dx=\bar \omega(0),\; a.s.$$
\end{prop}

Here, $\bar \omega$ is the differential 2-form integrated over the space and {\rm {d} denotes}  differential in the phase space taken with respect to the initial data.
Although the convergence result in Corollary \ref{cor-entropy} gives the approximation error between the entropy of $u^{\epsilon}$ and that of $u,$  it is still unknown what is the error between the regularization energy and the original energy due to loss of regularity in both time and space for Eq.  \eqref{SlogS} and Eq. \eqref{Reg-SlogS}. To overcome this issue, the weak convergence analysis on the regularization energy is introduced.

\begin{prop}\label{reg-ene}
	Let Assumption \ref{main-reg-fun} and the condition of Theorem \ref{mild-general} hold. Let $u^{0}$ and $u^\epsilon$ be the mild solutions of Eq. \eqref{SlogS} and Eq.  \eqref{Reg-SlogS}, respectively. Assume in addition that $\widetilde g(x)=\mathbf i x$ and that 
	\begin{align}\label{add-f}
	|1-\frac {\partial f_{\epsilon}} {\partial x}(x)x|\le C\epsilon\frac {1+\epsilon x+ x^2}{(\epsilon+x)(1+\epsilon x)}\quad{ \forall~x>0.}
	\end{align}
	Then the regularized energy is  convergent to the energy of Eq.  \eqref{SlogS}. 
	Furthermore, for $p\ge1$ and $\delta\in \Big(0,\max(\frac 2{\max(d-2,0)},1)\Big),$  there exist {$C(Q,T,\lambda,p,u_0,\delta)>0$ and $C(Q,T,\lambda,p,u_0,\alpha,\delta)>0$ such that}
	for $t\in [0,T],$ when $\mathcal O$ is a bounded domain,
	\begin{align*} 
	\E\Big[H_{\epsilon}(u^{\epsilon}(t))-H(u^0(t))\Big] \le C(Q,T,\lambda,p,u_0,\delta)(\epsilon^{\frac 12}+\epsilon^{\frac {{ \delta}} 2}),
	\end{align*}
	and when $\mathcal O=\mathbb R^d,$
	\begin{align*}
	\E\Big[H_{\epsilon}(u^{\epsilon}(t))-H(u^0(t))\Big] 
	&\le C(Q,T,\lambda,p,u_0,\alpha,\delta)(\epsilon^{\frac {\alpha }{2\alpha+d}}+\epsilon^{\frac {{ \delta}} 2}).
	\end{align*}
\end{prop}

\begin{proof}
	Applying the It\^o formula to $H_{\epsilon}(u^{\epsilon}(t))$ yields that 
	\begin{align*}
	H_{\epsilon}(u^{\epsilon}(t))=&H_{\epsilon}(u_0)
	+\int_0^t\<-\Delta u^{\epsilon},\widetilde g (u^{\epsilon})dW(s)\>+\int_0^t\frac 12\sum_{i\in \mathbb N^+}\|\widetilde g (u^{\epsilon})\nabla Q^{\frac 12}e_i\|^2ds
	\\
	&-\lambda\int_0^t \Big\< f_{\epsilon}(|u^{\epsilon}|^2)u^{\epsilon},\widetilde g (u^{\epsilon})dW(s)\Big\>\\
	&
	-\frac 12 \lambda \int_0^t\sum_{i\in \mathbb N^+}\Big\<2 Re(\bar u^{\epsilon}\widetilde g (u^{\epsilon})Q^{\frac 12}e_i)\frac {\partial f_{\epsilon}}{\partial x}(|u^{\epsilon}|^2)u^{\epsilon},\widetilde g (u^{\epsilon})Q^{\frac 12}e_i\Big\>ds.
	\end{align*}
	Then subtracting $H_{\epsilon}(u^{\epsilon}(t))$ from $H(u^0(t))$ and taking expectation, we get 
	\begin{align*}
	&\E\Big[H(u^0(t))-H_{\epsilon}(u^{\epsilon}(t))\Big]
	= \E\Big[H(u^0(0))-H_{\epsilon}(u^{\epsilon}(0))\Big]
	\\
	+&\int_0^t\frac 12\sum_{i\in \mathbb N^+}{\mathbb E}\Big(\|\widetilde g (u^{0})\nabla Q^{\frac 12}e_i\|^2-\|\widetilde g (u^{\epsilon})\nabla Q^{\frac 12}e_i\|^2\Big)ds\\
	&
	-\frac 12 \lambda \int_0^t\sum_{i\in \mathbb N^+}\mathbb E\Big[\Big\<2 Re(\bar u^{0}\widetilde g (u^{0})Q^{\frac 12}e_i)\frac {\partial f_{0}}{\partial x}(|u^{0}|^2)u^{0},\widetilde g (u^{0})Q^{\frac 12}e_i\Big\>\\
	&\quad-\Big\<2 Re(\bar u^{\epsilon}\widetilde g (u^{\epsilon})Q^{\frac 12}e_i)\frac {\partial f_{\epsilon}}{\partial x}(|u^{\epsilon}|^2)u^{\epsilon},\widetilde g (u^{\epsilon})Q^{\frac 12}e_i\Big\>\Big]ds\\
	=&: \frac {\lambda} 2\E [F_{\epsilon}(\rho^{\epsilon}(t))-F(\rho^{0}(t)]+ I_1+I_2.
	\end{align*}
	The H\"older inequality yields  
	\begin{align*}
	|I_1|&\le \int_0^t\frac 12 \Big|\sum_{i\in \mathbb N^+}\E\Big[\int_{\mathcal O}|\nabla Q^{\frac 12}e_i|^2(|u^{\epsilon}|^2-|u^0|^2)dx\Big]\Big|ds\\
	& \le \frac 12 {\sum_{i\in \mathbb N^+}\|\nabla Q^{\frac 12}e_i\|_{L^{\infty}}^2} T\sup_{t\in[0,T]}\E\Big[\|u^{\epsilon}{(t)}-u^0(t)\|\|u^{\epsilon}(t)+u^0(t)\|\Big].
	\end{align*}
	The fact that $\widetilde g(x)=\mathbf i x$ and \eqref{add-f}, together with the H\"older inequality and \eqref{gn-sob}, imply that for $\delta \in (0,1]$, $\alpha>\frac {d\eta}{2-2\eta}, \alpha\in (0,1]$,
	\begin{align*}
	|I_2|&\le \int_0^t \Big|\sum_{i\in \mathbb N^+}\E\Big[\int_{\mathcal O} |Im (Q^{\frac 12}e_i)|^2(|u^0(s)|^2-\frac {\partial f_{\epsilon}}{\partial x}(|u^{\epsilon}(s)|^2)|u^{\epsilon}(s)|^4)dx\Big]\Big|ds\\
	&\le CT|\lambda|\sum_{i\in \mathbb N^+}\| Q^{\frac 12}e_i\|_{L^{\infty}}^2\sup_{t\in[0,T]}\E\Big[  \|u^0(t)-u^{\epsilon}(t)\|\|u^{\epsilon}(t)+u^0(t)\| \Big]\\
	&\quad + C T|\lambda| \sum_{i\in \mathbb N^+}\| Q^{\frac 12}e_i\|_{L^{\infty}}^2\epsilon^{\delta} {\Big(1+\sup_{t\in [0,T]}\E\Big[\|u^{\epsilon}(t)\|^{2+2\delta}\Big]\Big)}\\
	&\quad +
	C T|\lambda| \sum_{i\in \mathbb N^+}\| Q^{\frac 12}e_i\|_{L^{\infty}}^2 \epsilon^{\eta}\sup_{t\in [0,T]}\E\Big[\|u^{\epsilon}(t)\|_{L^{2-2\eta}}^{2-2\eta}\Big].
	\end{align*}
	Combining the above estimates with \eqref{wei-sob}, Lemma \ref{lm-con} and Corollary \ref{cor-entropy}, we complete the proof.
\end{proof}

\section{Structure-preserving regularized  Lie--Trotter type splitting method}
In this section, we propose the regularized Lie--Trotter type splitting methods based on Eq. \eqref{Reg-SlogS},  and investigate their strong convergence rates and convergence rates of the regularized entropy. 

When considering Eq. \eqref{Reg-SlogS} driven by additive noise,  i.e., $\widetilde g=1$, we apply the following decomposition
\begin{align*}
dv&= \bi \Delta vdt+dW(t),\;
v=v_0,\\
dw&= \bi \lambda f_{\epsilon}(|w|^2)wdt,\;
w=w_0,
\end{align*}
whose flow satisfies
\begin{align*}
&v_{A,\mathcal F_0}(v_0,t)=\Phi_{A,\mathcal F_0}^t(v_0)=e^{\bi \Delta t}v_0+\int_0^te^{\bi \Delta (t-s)}dW(s),\\ &w_f({w_0},t)=\Phi_f^t(w_0)=w_0e^{\bi \lambda f_{\epsilon}(|w_0|^2)t}
\end{align*} 
with $e^{\bi\Delta t}$ being the $C_0$-group generated by $\bi \Delta$, where $v_0,w_0$ are $\mathcal F_0$-measurable. 	
Denote the time step size by $\tau$ such that $t_k=k\tau,$ $k=0,1,\cdots,N,$ and $T=N\tau.$ 
The splitting scheme for the additive noise case is defined by 
\begin{align}\label{spl-add}
u^{\epsilon}_{k+1}=\Phi^{\tau}_{k}(u^{\epsilon}_k)=\Phi_{A,\mathcal F_{t_k}}^{\tau}(\Phi_f^{\tau}(u^{\epsilon}_k)), \; k\ge0, \; u^{\epsilon}_0=u_0,\; \tau>0.
\end{align}

For the multiplicative noise case, i.e., $\widetilde g(x)=\mathbf i g(|x|^2)x$, we proposed two different splitting strategies. 
When  $W(t)$ is $\mathbb H$-valued, we use 
\begin{align*}
dv&= \bi \Delta vdt+\bi g(|v|^2)v\star dW(t),\;
v=v_0,\\
dw&= \bi \lambda f_{\epsilon}(|w|^2)wdt,\;
w=w_0.
\end{align*} 
The flow of the first subsystem is
\begin{align*}
&v_{M,\mathcal F_0}(v_0,t)=\Phi_{M,\mathcal F_0}^t(v_0)\\
=&e^{\bi \Delta t}v_0+\int_0^t e^{\bi \Delta (t-s)}\Big(-\frac 12\sum_{i\in\mathbb N^+}|Q^{\frac 12}e_i|^2\Big(|g(|\Phi_{M,\mathcal F_0}^s(v_0)|^2)|^2 \Phi_{M,\mathcal F_0}^s (v_0)\Big)\\
&-\bi \sum_{i \in \mathbb N^+}Im(Q^{\frac 12}e_i) Q^{\frac 12}e_i g(|\Phi_{M,\mathcal F_0}^s (v_0)|^2)g'(|\Phi_{M,\mathcal F_0}^s (v_0)|^2) |\Phi_{M,\mathcal F_0}^s (v_0)|^2
\Phi_{M,\mathcal F_0}^s(v_0)\Big)ds\\
&+\bi \int_0^te^{\bi \Delta (t-s)} g(|\Phi_{M,\mathcal F_0}^s (v_0)|^2)\Phi_{M,\mathcal F_0}^s(v_0) d W(s),
\end{align*}
and 
the splitting scheme is formulated as 
\begin{align}\label{spl-mul-com}
u^{\epsilon}_{k+1}=\Phi^{\tau}_k(u^{\epsilon}_k)= \Phi_{M,\mathcal F_{t_k}}^{\tau}(\Phi_f^{\tau}(u^{\epsilon}_k)), \; k\ge0, \; u^{\epsilon}_0=u_0,\; \tau>0.
\end{align}
For the purpose of performing the numerical method, one may use the exponential Euler method to further approximate $ \Phi_{M,\mathcal F_{t_k}}^{\tau}$ and denote 
\begin{align}\label{spl-mul-com1}
u^{\epsilon}_{k+1}=\Phi^{\tau}_k(u^{\epsilon}_k)=\widetilde \Phi_{M,\mathcal F_{t_k}}^{\tau}(\Phi_f^{\tau}(u^{\epsilon}_k)), \; k\ge0, u^{\epsilon}_0=u_0,\; \tau>0,
\end{align}
where 
\begin{align*}
\widetilde \Phi_{M,\mathcal F_{t_k}}^{\tau}(v_0):&=e^{\bi \Delta \tau}v_0+
\int_0^\tau e^{\bi \Delta \tau }\Big(-\frac 12\sum_{i \in\mathbb N^+}|Q^{\frac 12}e_i|^2\Big(|g(|v_0|^2)|^2v_0\Big)\\
&\quad-\bi \sum_{i\in \mathbb N^+}Im(Q^{\frac 12}e_i) Q^{\frac 12}e_i {g(|v_0|^2)}g'(|v_0|^2) |v_0|^2
v_0\Big)ds\\
&\quad+\bi \int_0^{\tau}e^{\bi \Delta (\tau-s)} g(|v_0|^2)v_0 d\widetilde W_k(s),
\end{align*} 
where $\widetilde W_k(s):= W(t_k+s)-W(t_k)$ and $v_0$ is $\mathcal F_{t_k}$-measurable.
One could also use other discretization, like $\int_0^{\tau}e^{\bi \Delta \tau} g(|v_0|^2)v_0 d\widetilde W_k(s)$, for the stochastic integral.

For \eqref{Reg-SlogS} driven by conservative multiplicative noise, i.e., ${W(t)}$ is $L^2(\mathcal O;\mathbb R)$-valued, we have another splitting strategy, that is,
\begin{align*}
dv&= \bi \Delta vdt,\;
v=v_0,\\
dw&= \bi \lambda f_{\epsilon}(|w|^2)wdt+\bi g(|w|^2)w\star dW(t),\;
w=w_0.
\end{align*}
We also remark that in this case, the stochastic integral $\bi g(|w|^2)w\star dW(t)$ is the Stratonovich integral.
Denote the flow of the corresponding subsystems by
\begin{align*}
v_D(v_0,t)&=\Phi_D^t(v_0)=e^{\bi \Delta t}v_0,\\
w_{f+g,\mathcal F_0}({w_0},t)&=\Phi_{f+g,\mathcal F_0}^t(w_0)=w_0e^{\bi \lambda f_{\epsilon}(|w_0|^2)t+{\bi}g(|w_0|^2)W(t)}
\end{align*}
{ with $\mathcal F_0$-measurable $v_0,w_0.$}
Then the splitting scheme is defined by
\begin{align}\label{spl-mul-rea}
u^{\epsilon}_{k+1}=\Phi^{\tau}_k(u^{\epsilon}_k)= \Phi_D^{\tau}( \Phi_{f+g,\mathcal F_{t_k}}^{\tau}(u^{\epsilon}_k)), \; k\ge0, \; u^{\epsilon}_0=u_0,\; \tau>0.
\end{align}
Following the above strategies, one may construct different kinds of splitting methods by changing the order of the splitting or making a composition of different subsystems. 

\subsection{Structure-preserving properties  of regularized  Lie--Trotter type splitting method}

By making use of the properties of subsystems in the splitting methods, we present the following structure-preserving properties for the proposed numerical methods. The proof of the following propositions is omitted since it is similar to that of the exact solution of \eqref{Reg-SlogS}.

\begin{prop}
	Let Assumption \ref{main-reg-fun} and the condition of Theorem \ref{mild-general} hold. Assume in addition that $\{W(t)\}_{t\ge 0}$ is $\mathbb H$-valued, $\widetilde g=1$ or that  $\{W(t)\}_{t\ge 0}$ is $L^2(\mathcal O;\mathbb R)$-valued, $\widetilde g(x)=\mathbf ig(|x|^2)x.$ Then the phase flows of the splitting methods \eqref{spl-add}, \eqref{spl-mul-com} and \eqref{spl-mul-rea} preserve the symplectic structure, i.e.,
	$$\bar \omega_{k+1}:=\int_{\mathcal O} {\rm d}P_{k+1}^\epsilon \wedge {\rm d}Q_{k+1}^\epsilon dx=\int_{\mathcal O} {\rm d}P_{k}^\epsilon \wedge {\rm d}Q_{k}^\epsilon dx=
	\bar \omega_k, \; a.s.,$$
	with $P_k^\epsilon$ (resp., $Q_k^\epsilon$) being the real (resp., imaginary) part of the solution $u^{\epsilon}_k.$
\end{prop}

\begin{prop}
	Let Assumption \ref{main-reg-fun} and the condition of Theorem \ref{mild-general} hold. Assume that $\widetilde g(x)=\mathbf i x$ or $\widetilde g(x)=1$.Then the splitting methods \eqref{spl-add}, \eqref{spl-mul-com} and \eqref{spl-mul-rea} preserve the evolution law of the mass of the exact solution ${u^\epsilon}$, i.e.,
	\begin{align*}
	\E\Big[M(u_{k+1}^{\epsilon})\Big]=\E\Big[M(u_{k}^{\epsilon})\Big]+\tau \sum_{i\in \mathbb N^+}\|Q^{\frac 12}e_i\|^2 \chi_{\{\widetilde g=1\}},
	\end{align*}
	where $\chi_{\{\widetilde g=1\}}=1$ for the additive noise case and $\chi_{\{\widetilde g=1\}}=0$ for the multiplicative noise case.
\end{prop}

\subsection{Strong convergence rate of Lie--Trotter type splitting method for RSlogS equation}

In the following, we give the strong convergence analysis of the proposed splitting methods.

\begin{prop}\label{tm-con-add}
	Let Assumption \ref{main-reg-fun} and the condition of Theorem \ref{mild-general} hold. Let $\widetilde g=1.$ Assume in addition that $f_{\epsilon}$ satisfies 
	\begin{align}\label{con-f}
	|f_{\epsilon}(|x|^2)x-f_{\epsilon}(|y|^2)y|&\le C(1+|\log(\epsilon)|)|x-y|.
	\end{align}
	Then the numerical solution of the splitting method \eqref{spl-add} is strongly convergent to the exact one of Eq. \eqref{Reg-SlogS}. Moreover, for $p\ge2,$  there exists $C(Q,T,\lambda,p,u_0)>0$ such that
	\begin{align*}
	\sup_{k\le N}\|u^{\epsilon}_k-u^{\epsilon}(t_k)\|_{L^{p}(\Omega;\mathbb H)}&\le C(Q,T,\lambda,p,u_0)(1+|\log(\epsilon)|)\tau^{\frac 12}.
	\end{align*}
\end{prop}

\begin{proof}
	
	For convenience, we illustrate the procedures in the case that $p=2.$ For general $p\ge 2$, the proof is similar. 
	Assume that $v\in \mathbb H^1$ is $\mathcal F_{t_k}$-measurable and has any finite moment. 
	Denote the exact flow of \eqref{Reg-SlogS} on $[t_k,t_{k+1}]$ by $\Psi_k^t, t\in [0,\tau]$  and the numerical flow by $\Phi^t_k,  t\in [0,\tau].$ Then the equation of $\Psi^t_k(v)$ and $\Phi^t_k(v)=\Phi_{A,\mathcal F_{t_k}}^{t}(\Phi_f^{t}(v))$ on a small interval $[0,\tau]$ can be rewritten as
	\begin{align*}
	d\Psi_k^t(v)&=\bi \Delta \Psi_k^t(v)dt+ d\widetilde W_k(t)+\bi \lambda f_{\epsilon}(|\Psi_k^t(v)|^2)\Psi_k^t(v)dt,\\
	d\Phi_k^t(v)&=\bi \Delta \Phi_k^t(v)dt+ d\widetilde W_k(t)+\bi \lambda e^{\bi \Delta t} [f_{\epsilon}(|\Phi_f^t(v)|^2)\Phi_f^t(v)]dt.
	\end{align*}
	Let  $\varepsilon^t_k(v)=\Psi^t_k(v)-\Phi^t_k(v).$ Then it holds that  
	\begin{align*}
	d\varepsilon^t_k(v)={\bi \Delta \varepsilon^t_k(v)dt}+\bi \lambda  \Big(f_{\epsilon}(|{\Psi^t_k}(v)|^2){\Psi^t_k}(v)-e^{\bi \Delta t} [f_{\epsilon}(|\Phi_f^t(v)|^2)\Phi_f^t(v)]\Big)dt.
	\end{align*}
	Applying the chain rule, integration by parts and  (A2), we achieve that 
	\begin{align*}
	\frac {d}{dt} \|\varepsilon^t_k(v)\|^2
	&\le 2C\|{\varepsilon^t_k}(v)\|^2
	+2C\|{\varepsilon^t_k}(v)\|\Big\|f_{\epsilon}(|{\Phi^t_k}(v)|^2){\Phi^t_k}(v)-f_{\epsilon}(|\Phi_f^t(v)|^2)\Phi_f^t(v)\Big\|\\
	&\quad+2C\|\varepsilon_k^t(v)\|\|(I-e^{\bi \Delta t})[f_{\epsilon}(|\Phi_f^t(v)|^2)\Phi_f^t(v)]\|\\
	&=: 2C\|{\varepsilon_k^t}(v)\|^2+2C\|{\varepsilon_k^t}(v)\|I_1+2C\|{\varepsilon_k^t}(v)\|I_2.
	\end{align*}
	This leads to $\frac d{dt}\|{\varepsilon_k^t}(v)\|\le C\Big(\|{\varepsilon^t_k}(v)\|+I_1+I_2\Big).$
	The property \eqref{con-f} of $f_{\epsilon},$  together with the definition of $\Phi_k$ and  $\|w-e^{\bi \Delta t}w\|\le C\sqrt{t}\|w\|_{\mathbb H^1}$, yields that  
	\begin{align*}
	I_1
	&\le C(1+|\log(\epsilon)|)\|\Phi_{A,\mathcal F_{t_k}}^t{(\Phi_f^t(v))}-\Phi_f^t(v)\|\\
	&\le C(|\log(\epsilon)|+1)\sqrt{t}\|v\|_{\mathbb H^1}+C(1+|\log(\epsilon)|)\Big\|\int_0^te^{\bi \Delta (t-s)}d\widetilde W_k(s)\Big\|.
	\end{align*}
	In the last inequality,  we use the following estimate
	\begin{align*}
	\|\Phi_f^s(v)\|_{\mathbb H^1}^2&\le \|v\|_{\mathbb H^1}^2+\|v\|^2+4\lambda^2s^2\Big\|v\frac {\partial f}{\partial x}(|v|^2) Re(\bar v\nabla v)\Big\|^2\\
	&\le \|v\|_{\mathbb H^1}^2+\|v\|^2+4C|\lambda|^2s^2\|\nabla v\|^2\le \|v\|_{\mathbb H^1}^2+C s^2\|v\|_{\mathbb H^1}^2,\; s\in [0,\tau].
	\end{align*} 
	For the term $I_2$, similar arguments lead to
	\begin{align*}
	I_2
	&\le C\sqrt{t}\Big\|f_{\epsilon}(|\Phi_f^t(v)|^2)\Phi_f^t(v)\Big\|_{\mathbb H^1}\le C\sqrt{t}\Big((1+|\log(\epsilon)|)(1+Ct)\|v\|_{\mathbb H^1}\Big).
	\end{align*}
	Combining the above estimates on $I_1$ and $I_2$, we have that 
	\begin{align*}
	\|{\varepsilon_k^t}(v)\|\le C(|\log(\epsilon)|+1)\int_0^t\left(\sqrt{s}\|v\|_{\mathbb H^1}+\left\|\int_0^se^{\bi \Delta (s-r)}d\widetilde W_k(r)\right\|\right)ds.
	\end{align*} 
	This, together with the Burkholder inequality, implies that 
	\begin{align*}
	\|{\varepsilon_k^t}(v)\|_{L^p(\Omega;\mathbb H)}&\le t^{\frac 32}C(1+|\log(\epsilon)|)(1+\|v\|_{L^{p}(\Omega;\mathbb H^1)}).
	\end{align*}
	
	Next we show the stability of ${\Phi^t_k}$ to get the global error 
	estimate.
	Direct calculations, together with the chain rule and $(A 2)$, yield that 
	\begin{align*}
	&\|\Phi_f^t(v)-\Phi_f^t(w)\|^2\\
	=&\|v-w\|^2+2\int_0^t\<\Phi_f^s(v)-\Phi_f^s(w),f_{\epsilon}(|\Phi_f^s(v)|^2)\Phi_f^s(v)- f_{\epsilon}(|\Phi_f^s(w)|^2)\Phi_f^s(w)\>ds\\
	\le& \|v-w\|^2+4\int_0^t\|\Phi_f^s(v)-\Phi_f^s(w)\|^2ds.
	\end{align*}
	Therefore, it holds that ${\|\Phi_f^t(v)-\Phi_f^t(w)\|^2}\le \exp(Ct)\|v-w\|^2.$
	This, together with the fact that $\widetilde g=1,$ implies
	$
	\|{\Phi^t_k}(v)-\Phi^{t}_k(w)\|\le \exp(Ct)\|v-w\|$.
	The similar arguments lead to the stability estimate of ${\Phi^t_k}$,
	\begin{align*}
	\|{\|{\Phi^t_k}(v)\|^2\|}_{L^p(\Omega)}&\le \exp(Ct)(1+\|v\|_{L^{2p}(\Omega;\mathbb H)}^2),\\
	\|\|{\Phi^t_k}(v)\|_{\mathbb H^1}^2\|_{L^p(\Omega)} &\le \exp(Ct)(1+\|v\|_{L^{2p}(\Omega;\mathbb H^1)}^2).
	\end{align*}
	Taking $p$-moment and using stability estimates of $\Phi_k$ and ${\Psi_k}$, and decomposing the global error as 
	\begin{align*}
	\|u^{\epsilon}(t_N)-u^{\epsilon}_N\|
	&\le \sum_{k=0}^{N-1}\Big\|\prod_{j={k+1}}^{N-1}\Psi^{\tau}_j \Big(\Psi_k^{\tau}-\Phi_k^{\tau}\Big)\Big(\prod_{j=0}^{k-1}\Phi^{\tau}_j (u_0)\Big)\Big\|\\
	&\le C\sum_{k=0}^{N-1}\tau^{\frac 32}(1+|\log(\epsilon)|)(1+\|\prod_{j=0}^{k-1}\Phi^{\tau}_j (u_0) \|_{\mathbb H^1}),
	\end{align*}
	we complete the proof.
\end{proof}

The method of proving the strong convergence in additive noise case can not be directly used to the multiplicative noise case due to the existence of diffusion terms.  Below we present the convergence analysis of  the proposed methods in the multiplicative noise case. 

\begin{prop}\label{tm-con-mul-com}
	Let Assumption \ref{main-reg-fun} and the condition of Theorem \ref{mild-general} hold. Let $\widetilde g(x)=\mathbf i g(|x|^2)x.$ Assume in addition that $f_{\epsilon}$ satisfies \eqref{con-f}. 
	Then the numerical solution of \eqref{spl-mul-com} is strongly convergent to the exact solution of Eq. \eqref{Reg-SlogS}. Moreover, for $p\ge2,$  there exists $C(Q,T,\lambda,p,u_0,\widetilde g)>0$ such that
	\begin{align*}
	\sup_{k\le N}\|u^{\epsilon}_k-u^{\epsilon}(t_k)\|_{L^{p}(\Omega;\mathbb H)}&\le C(Q,T,\lambda,p,u_0,\widetilde g)(1+|\log(\epsilon)|)\tau^{\frac 12}.
	\end{align*}
\end{prop}

\begin{proof}
	We show the details of the proof in the case that $p=2.$  
	Assume that $v,w \in \mathbb H^1$ are $\mathcal F_{t_k}$-measurable and have any finite moment. 
	Denote the exact flow of \eqref{Reg-SlogS} on $[t_k,t_{k+1}]$ by $\Psi_k^s,$ $s\in[0,\tau]$.
	Fix $t \in [0,\tau],$ define the auxiliary flow  $\widetilde \Phi^s_k(w)=\Phi_{M,\mathcal F_{t_k}}^{s}(\Phi_f^{t}(w))$, $s\in [0,\tau].$
Then the equations of ${\Psi^t_k}(v)$ and $ \Phi^{t}_k (w)$ and can be rewritten as 
	
	\begin{align*}
	d{\Psi^t_k}(v)=&\bi \Delta {\Psi^t_k}(v)dt+\bi g(|{\Psi^t_k}(v)|^2){\Psi^t_k}(v)d\widetilde W_k(t)\\
	&-\bi g'(|{\Psi^t_k}(v)|^2)g(|{\Psi^t_k}(v)|^2)|{\Psi^t_k}(v)|^2{\Psi^t_k}(v)\sum_{i\in \mathbb N^+} Im( Q^{\frac 12}e_i) Q^{\frac 12}e_idt\\
	&-\frac 12(g(|{\Psi^t_k}(v)|^2))^2{\Psi^t_k}(v)\sum_{i\in \mathbb N^+}|Q^{\frac 12}e_i|^2dt+\bi \lambda f_{\epsilon}(|{\Psi^t_k}(v)|^2){\Psi^t_k}(v)dt,\\
	d{\Phi^t_k}(w)=&\bi \Delta {\Phi^t_k}(w)dt+\bi g(|{\widetilde \Phi^t_k}(w)|^2){\widetilde  \Phi^t_k}(w)d \widetilde W_k(t)\\
	&-\bi g'(|{\widetilde  \Phi^t_k}(w)|^2)g(|{\widetilde  \Phi^t_k}(w)|^2)|{\widetilde  \Phi^t_k}(w)|^2{\widetilde  \Phi^t_k}(w)\sum_{i\in \mathbb N^+}  Im( Q^{\frac 12}e_i) Q^{\frac 12}e_idt\\
	&-\frac 12(g(|{\widetilde  \Phi^t_k}(w)|^2))^2{\widetilde  \Phi^t_k}(w)\sum_{i\in \mathbb N^+}|Q^{\frac 12}e_i|^2dt+\bi \lambda e^{\bi \Delta t} [f_{\epsilon}(|\Phi_f^t(w)|^2)\Phi_f^t(w)]dt.
	\end{align*}	
	Using the It\^o formula, letting $v=u^{\epsilon}(t_k)$ and $w=u^{\epsilon}_k$, taking expectation and using the conditions \eqref{con-g1} and \eqref{con-g} on $g$, as well as  the condition (A2) on $f_{\epsilon}$, we have that 

		\begin{align*}
		&\E\Big[\|\Psi^{t}_k(v)-\Phi^{t}_k(w)\|^2\Big]\\
		\le &\E\Big[\|v-w\|^2\Big]+
		C\int_0^t\E\Big[\|\Psi^{s}_k(v)-\Phi^{s}_k(w)\|^2\Big]ds+C\int_0^t\E\Big[\|\widetilde \Phi^{s}_{k}(w)-\Phi^{s}_k(w)\|^2\Big]ds\\
		&+{C\int_0^t \E\Big[\Big\|f_{\epsilon}(|\Phi^s_k(w)|^2)\Phi^s_k(w)-e^{\bi \Delta s}f_{\epsilon}(|\Phi_f^s(w)|^2)\Phi_f^s(w)\Big\|^2\Big]ds.} 
		\end{align*}		
Using the conditions \eqref{con-g1} and \eqref{con-g} on $g$ and It\^o's formula, as well as $\|w-e^{\bi \Delta t}w\|\le C\sqrt{t}\|w\|_{\mathbb H^1}$, we arrive at that 
\begin{align*}
\E\Big[\|\widetilde \Phi^{s}_{k}(w)-\Phi^{s}_k(w)\|^2\Big]
&\le C\E\Big[\|\Phi_f^{t}(w)-\Phi_f^{s}(w)\|^2\Big]\\
&\le C \tau^2\sup_{s\in[0,t]}(1+|\log(\epsilon)|^2)\E\Big[\|\Phi^s_f(w)\|^2\Big].
\end{align*}
It suffices to estimate 
	\begin{align*}
	&\Phi^{s}_{M,\mathcal F_{t_k}} (\Phi_f^s(w))-\Phi_f^s(w)
	=
	(e^{\bi \Delta t}-I)\Phi_f^s(w)
	+\int_{0}^s e^{\bi \Delta{ (s-r)}} II_{mod,k}^r(\Phi_f^s(w))dr\\
	&\qquad \qquad\qquad\qquad \qquad\quad +\int_0^s e^{\bi \Delta( s-r)} g(|\Phi_{M,\mathcal F_k}^r (\Phi_f^s(w))|^2)\Phi_{M,\mathcal F_k}^r(\Phi_f^s(w)) d\widetilde W_k(r)
	\end{align*}
	for $s\in [0,\tau],$
where 
\begin{align*}
&II_{mod,k}^r(v_0)
:= -\frac 12\sum_{i\in\mathbb N^+}|Q^{\frac 12}e_i|^2|g(|\Phi_{M,\mathcal F_k}^r(v_0)|^2)|^2 \Phi_{M,\mathcal F_k}^r (v_0)\\
&-\bi \sum_{i\in \mathbb N^+}Im(Q^{\frac 12}e_i) Q^{\frac 12}e_ig(|\Phi_{M,\mathcal F_k}^r (v_0)|^2)g'(|\Phi_{M,\mathcal F_0}^r (v_0)|^2) |\Phi_{M,\mathcal F_k}^r (v_0)|^2
\Phi_{M,\mathcal F_k}^r(v_0),
\end{align*}
for any $\mathcal F_{k}$-measurable function $v_0$.
	By taking second moment and using the continuity estimate of $e^{\bi \Delta t}$ and the boundedness of the flow $\Phi_{M,\mathcal F_k}$, i.e. $\E[\|\Phi_{M,\mathcal F_{t_k}}^r(v_0)\|_{\mathbb H^1}^2]\le C\|v_0\|_{\mathbb H^1}^2$,
	we get 
	\begin{align*}
	&\E\Big[\| \Phi^{s}_{M,\mathcal F_{t_k}} (\Phi_f^s(w))-\Phi_f^s(w)\|^2\Big]\\
	\le& 2C\tau \E\Big[\|\Phi_f^s(w)\|_{{\mathbb H^1}}^2\Big]+2\E\Big[\int_0^s \|II_{mod,k}^s({\Phi_f^s(w)})\|^2dr\Big]\\
	&+2\E\Big[\|\int_0^s e^{\bi \Delta (s-r)}g(|\Phi_{M,\mathcal F_k}^r (\Phi_f^s(w))|^2)\Phi_{M,\mathcal F_k}^r(\Phi_f^s(w))  d{\widetilde W_k}(r)\|^2\Big]\\
	\le& C\tau \Big(1+\E[\|\Phi_f^{s} (w)\|_{\mathbb H^1}^2\Big]\Big).
	\end{align*}
	Using the Gronwall inequality and similar arguments in proving the estimates of $I_1$ and $I_2$ in the proof of  Proposition \ref{tm-con-add}, we have 
	{\small
		\begin{align*}
		&\E\Big[\|\Psi^{t}_k(v)-\Phi^{t}_k(w)\|^2\Big]\le e^{C\tau}\Big(\E\Big[\|v-w\|^2\Big]+C\tau^2(1+|\log(\epsilon)|^2)\sup_{s\in[0,t]}\E\Big[\|\Phi^s_f(w)\|_{\HH^1}^2\Big]\Big).
		\end{align*}
		B}ased on the definition of $\Phi^t_f$ and ${\Phi^t_k}$, the stability estimate in $\mathbb H^1$
	\begin{align*}
	\|\Phi^t_f(w)\|_{\HH^1}^2&\le \|w\|_{\mathbb H^1}^2+C{ t^2}\|w\|_{\mathbb H^1}^2,\\
	\E[\|\Phi^t_k(w)\|_{\HH^1}^2]&\le \exp(Ct) \E\Big[\|w\|_{\mathbb H^1}^2+C { t^2}\|w\|_{\mathbb H^1}^2\Big]
	\end{align*}
	can be shown. 
	Taking $t=\tau,$ and using the a priori estimate of $\|\Phi_k^s(w)\|_{\HH^1}$ and $\|\Phi^s_f(w)\|_{\HH^1}$, we have 
	\begin{align*}
	&\E\Big[\|u^{\epsilon}(t_{k+1})-u_{k+1}^{\epsilon}\|^2\Big]
	\le e^{C\tau}\Big(\E\Big[\|u^{\epsilon}(t_{k})-(u_{k}^{\epsilon})\|^2\Big]+C\tau^2(1+|\log(\epsilon)|^2)\Big).
	\end{align*}
	By repeating the above procedures, we conclude 
	\begin{align*}
	\E\Big[\|u^{\epsilon}(t_{k+1})-u_{k+1}^{\epsilon}\|^2\Big]
	&\le C(Q,T,\lambda,u_0,{p,\tilde g})(1+|\log(\epsilon)|^2)\tau,
	\end{align*}
	which completes the proof. 
\end{proof}

\begin{prop}\label{tm-con-com}
	Under the condition of Proposition  \ref{tm-con-mul-com},
	the splitting scheme \eqref{spl-mul-com1} is 
	strongly convergent. Moreover, for $p\ge2,$  there exists $C(Q,T,\lambda,p,u_0,\widetilde g)>0$ such that
	\begin{align*}
	\sup_{k\le N}\|u^{\epsilon}_k-u^{\epsilon}(t_k)\|_{L^{p}(\Omega;\mathbb H)}&\le C(Q,T,\lambda,p,u_0,\widetilde g)(1+|\log(\epsilon))|)\tau^{\frac 12}.
	\end{align*}
\end{prop}

\begin{proof}
	The proof is  similar to that of  Proposition \ref{tm-con-mul-com}. We present the details for $p=2.$
	The main difference is that for an $\mathcal F_{t_k}$-measurable $w$, ${\Phi^t_k}(w)$ is replaced by $\widehat \Phi^t_k(w)$  which satisfies 
	\begin{align*}
	d{\widehat \Phi^t_k}(w)=&\bi \Delta {\widehat \Phi^t_k}(w)dt+\bi e^{\bi \Delta t}g(|\Phi^{\tau}_f(w)|^2)\Phi^{\tau}_f(w)d{\widetilde W_k}(t)+{e^{\bi\Delta t}}II_{mod}({\Phi^{\tau}_f(w)})dt\\
	&+\bi \lambda e^{\bi \Delta t} [f_{\epsilon}(|\Phi_f^t (w)|^2)\Phi_f^t (w)]dt,
	\end{align*}
	where 
	\begin{align*}
	II_{mod}(\Phi^{\tau}_f(w)):=&-\frac 12(g(|{\Phi^{\tau}_f(w)}|^2))^2{\Phi^{\tau}_f(w)}\sum_{i\in \mathbb N^+}|Q^{\frac 12}e_i|^2\\
	&-\bi g'(|{\Phi^{\tau}_f(w)}|^2)g(|{\Phi^{\tau}_f(w)}|^2)|{\Phi^{\tau}_f(w)}|^2{\Phi^{\tau}_f(w)}\sum_{i\in \mathbb N^+} Im( Q^{\frac 12}e_i) Q^{\frac 12}e_i.
	\end{align*}
	Then we have that ${
		{\widehat \Phi}}^{\tau}_k(w)=\lim\limits_{t\to\tau}\widetilde \Phi^{t}_{M,\mathcal F_{t_k}}{(\Phi^{t}_f(w)).}$ 
	The following a priori estimates of the numerical schemes can be obtained,
	{
		\begin{align*}
		\sup_{k\le N-1}\sup_{t\in[0,\tau]}\big[ \|\Phi_f^t(u_{k}^{\epsilon})\|_{L^q(\Omega;\mathbb H^1)}+\|{
			{\widehat \Phi}}_k^t(u_{k}^{\epsilon})\|_{L^q(\Omega;\mathbb H^1)}\big] <\infty
		\end{align*} 
		for any $q\ge 2$, via similar steps in  the proof of Proposition \ref{tm-con-mul-com}.
		Using} the It\^o formula  and letting $v=u^{\epsilon}(t_k)$ and $w=u^{\epsilon}_k$, then taking expectation and {exploiting} the conditions \eqref{con-g1} and \eqref{con-g} on $g$, as well as  the condition (A2) on $f_{\epsilon}$,  yield that
		\begin{align*}
		&\E\Big[\|\Psi^{t}_k(v)-\widehat\Phi^{t}_k(w)\|^2\Big]\\
		\le &\E\Big[\|v-w\|^2\Big]+C\E\int_0^t \E \Big[\|\widehat \Phi^s_k(w)-\Phi_f^\tau(w)\|^2\Big]ds +
		C\int_0^t\E\Big[\|\Psi^{s}_k(v)-\widehat \Phi^{s}_k(w)\|^2\Big]ds\\
		+&C\int_0^t \E\Big[\Big\|f_{\epsilon}(|\widehat \Phi^s_k(w)|^2) \widehat  \Phi^s_k(w)-e^{\bi \Delta s}f_{\epsilon}(|\Phi_f^s(w)|^2)\Phi_f^s(w)\Big\|^2\Big]ds\\
		&+C\int_0^t \E\Big[\Big\| II_{mod}(\widehat \Phi^s_k(w))-e^{\bi \Delta s}II_{mod}(\Phi_f^\tau (w))\Big\|^2\Big]ds.
		\end{align*}
				
		Similar to the proof of Proposition \ref{tm-con-mul-com}, it suffices to estimate  $\widetilde \Phi^{t}_{M,\mathcal F_{t_k}} (\Phi_f^t(w))-\Phi_f^t(w)$ for  and  $\Phi_f^t(w)-\Phi_f^\tau (w)$ for $t\in [0,\tau]$.
By the definition of $\Phi_f$ and (A1), $\E \Big[ \| \Phi_f^t(w)-\Phi_f^\tau (w)\|^2\Big]\le C \tau^2(1+|\log(\epsilon)|^2) \|\Phi_f^t(w)\|^2.$		
	The definition of $\widetilde \Phi^{t}_{M,\mathcal F_{t_k}}$  yields that
	\begin{align*}
	&\widetilde \Phi^{t}_{M,\mathcal F_{t_k}} (\Phi_f^t(w))-\Phi_f^t(w)\\
	=&
	(e^{\bi \Delta t}-I)\Phi_f^t(w)
	+\int_{0}^t e^{\bi \Delta{ t}} II_{mod}(\Phi_f^\tau (w))ds+\int_0^t e^{\bi \Delta t} \bi g(|\Phi^{\tau}_f(w)|^2)\Phi^{\tau}_f(w)d\widetilde W_k(s).
	\end{align*}
	By taking the second moment and using the continuity estimate of $e^{\bi \Delta t},$
	we get 
	\begin{align*}
	&\E\Big[\|\widetilde \Phi^{t}_{M,\mathcal F_{t_k}} (\Phi_f^t(w))-\Phi_f^t(w)\|^2\Big]\\
	\le& 2\tau \E\Big[\|\Phi_f^t(w)\|_{{\mathbb H^1}}^2\Big]+2\E\Big[\int_0^t \|II_{mod}({\Phi_f^\tau (w)})\|^2ds\Big]\\
	&+2\E\Big[\|\int_0^te^{\bi \Delta t}g(|\Phi_f^{\tau} (w)|^2)\Phi^{\tau}_f(w)d{\widetilde W_k}(s)\|^2\Big]
	\le C\tau \Big(1+\E[\|\Phi_f^{t} (w)\|_{\mathbb H^1}^2\Big]\Big).
	\end{align*}
	Substituting the above estimates into the estimate $\E\Big[\|\Psi^{t}_k(v)-\widehat \Phi^{t}_k(w)\|^2\Big]$, and using the priori estimate of $\Phi_f^t(w)$ and $\widehat \Phi_k^t(w)$, we get that for $t=\tau,$ we have 
	\begin{align*}
	&\E\Big[\|u^{\epsilon}(t_{k+1})-u_{k+1}^{\epsilon}\|^2\Big]
	\le e^{C\tau}\Big(\E\Big[\|u^{\epsilon}(t_{k})-(u_{k}^{\epsilon})\|^2\Big]+C\tau^2(1+|\log(\epsilon)|^2)\Big).
	\end{align*}
	By repeating the above procedures, we conclude 
	\begin{align*}
	\E\Big[\|u^{\epsilon}(t_{k+1})-u_{k+1}^{\epsilon}\|^2\Big]
	&\le C(Q,T,\lambda,u_0,{p,\tilde g})(1+|\log(\epsilon)|^2)\tau,
	\end{align*}
	which completes the proof. 
\end{proof}

\begin{prop}\label{tm-con-real}
	Let the condition of Proposition \ref{tm-con-mul-com} hold. Assume that $W(t)$ is an $L^2(\mathcal O;\mathbb R)$-valued process. 
	Then 
	the splitting scheme \eqref{spl-mul-rea} is 
	strongly convergent. Moreover, for $p\ge2,$  there exists $ C(Q,T,\lambda,p,u_0,\widetilde g)>0$ such that
	\begin{align*}
	\sup_{k\le N}\|u^{\epsilon}_k-u^{\epsilon}(t_k)\|_{L^{p}(\Omega;\mathbb H)}&\le{ C(Q,T,\lambda,p,u_0,\widetilde g)(1+|\log(\epsilon)|)}\tau^{\frac 12}.
	\end{align*}
\end{prop}

\begin{proof}
	We only present the details of the case that $p=2.$  
	Assume that $v,w\in \mathbb H^1$ are $\mathcal F_{t_k}$-measurable and have any finite moment. 
	Denote the exact flow of Eq. \eqref{Reg-SlogS} on $[t_k,t_{k+1}]$ by $\Psi_k^t,$ $t\in[0,\tau]$, and the numerical flow of \eqref{spl-mul-rea} by $\Phi^t_k.$ 
	Then the equation of $\Phi^t_k(w)$ on a small interval $[0,\tau]$ can be rewritten as 
	\begin{align*}
	d{\Phi^t_k}(w)
	=&\bi \Delta {\Phi^t_k}(w)dt+\bi e^{\bi \Delta t} g(|\Phi_{f+g,\mathcal F_{t_k}}^t(w)|^2)\Phi_{f+g,\mathcal F_{t_k}}^t(w)d\widetilde W_k(t)\\
	&-\frac 12e^{\bi \Delta t}(g(|\Phi_{f+g,\mathcal F_{t_k}}^t(w)|^2))^2\Phi_{f+g,\mathcal F_{t_k}}^t(w)\sum_{i}|Q^{\frac 12}e_i|^2\\
	&+\bi \lambda e^{\bi \Delta t} [f_{\epsilon}(|\Phi_{f+g,\mathcal F_{t_k}}^t(w)|^2)\Phi_{f+g,\mathcal F_{t_k}}^t(w)]dt.
	\end{align*}
	Denote $v=u^{\epsilon}(t_k)$ and $w=u^{\epsilon}_k$.
	Following the similar procedures in the proof of Proposition \ref{tm-con-mul-com}, and
	using the chain rule, the growth condition on $g$, \eqref{con-g}, \eqref{con-f} and (A2), we get that for $t\in [0,\tau],$
	{\small
		\begin{align*}
		\E\Big[\|\Psi^{t}_k(v)-\Phi^{t}_k(w)\|^2\Big]
		\le& \E\Big[\|v-w\|^2\Big]+C\int_0^t\E\Big[\|\Psi_k^{s}(v)-\Phi^{s}_k(w)\|^2\Big]ds\\
		&+C\int_0^t\E\Big[\Big\|(I-e^{\bi \Delta s})f_{\epsilon}(|\Phi_{f+g,\mathcal F_{t_k}}^s(w)|^2)\Phi_{f+g,\mathcal F_{t_k}}^s(w)\Big\|^2\Big]ds\\
		&+C(1+|\log(\epsilon)|^2)\int_0^t\E\Big[\Big\|\Phi^s_{k}(w)-\Phi_{f+g,\mathcal F_{t_k}}^s(w)\Big\|^2\Big]ds\\
		&+C\int_0^t\E\Big[\|(I-e^{\bi\Delta s})(g(|\Phi_{f+g,\mathcal F_{t_k}}^s(w)|^2))^2\Phi_{f+g,\mathcal F_{t_k}}^s(w)\|^2\Big]ds.
		\end{align*}
		B}y the similar procedures in the proof of Proposition \ref{tm-con-mul-com}, it is not hard to obtain the a priori estimate of the numerical solution, 
	that is 
	$\sup_{k\le N} \E\Big[\|u_k^{\epsilon}\|_{\mathbb H^1}^p\Big] \Big]\le C.$
	The property of $e^{\bi\Delta s}$, the growth condition of $g$, (A1) and (A5), yield that 
	\begin{align*}
	&\E\Big[\Big\|\Phi^s_k(w)-\Phi_{f+g,\mathcal F_{t_k}}^s(w)\Big\|^2\Big]
	\le C\E\Big[\|\Phi_D^s\Phi_{f+g,\mathcal F_{t_k}}^s(w)-\Phi_{f+g,\mathcal F_{t_k}}^s(w)\|^2\Big]\\
	\le& C\tau \E\Big[\|\Phi_{f+g,\mathcal F_{t_k}}^s(w)\|_{\mathbb H^1}^2\Big]\le C\tau(1+|\log(\epsilon)|^2)\E\Big[\|w\|_{\mathbb H^1}^2\Big].
	\end{align*}
	and
	\begin{align*}
	&\E\Big[\Big\|(I-e^{\bi\Delta s})f_{\epsilon}(|\Phi_{f+g,\mathcal F_{t_k}}^s(w)|^2)\Phi_{f+g,\mathcal F_{t_k}}^s(w)\Big\|^2\Big]\\
	&+\E\Big[\|(I-e^{\bi\Delta s})(g(|\Phi_{f+g,\mathcal F_{t_k}}^s(w)|^2))^2\Phi_{f+g,\mathcal F_{t_k}}^s(w)\|^2\Big]
	\le C\tau(1+|\log(\epsilon)|^2)\E\Big[\|w\|_{\mathbb H^1}^2\Big].
	\end{align*}
	Then the Gronwall inequality, together with the above estimates, yields that 
	\begin{align*}
	\E\Big[\|{\Psi_k^{t}(v)-\Phi_k^{t}}(w)\|^2\Big]
	\le& e^{C\tau}\Big(\E\Big[\|v-w\|^2\Big]+\tau^2(1+|\log(\epsilon)|^2)\E\Big[{\|w\|_{\HH^1}^2}\Big]\Big)\\
	\le & e^{C\tau}\Big(\E\Big[\|v-w\|^2\Big]+\tau^2(1+|\log(\epsilon)|^2)\Big).
	\end{align*}
	Making use of an iteration argument and the a priori estimates of $u^{\epsilon}_k$, we obtain 
	\begin{align*}
	\E\Big[\|u^{\epsilon}_{k+1}-u(t_{k+1})\|^2\Big]
	&\le  e^{C\tau}\Big(\E\Big[\|u^{\epsilon}_{k}-u(t_{k})\|^2\Big]+C\tau^2(1+|\log(\epsilon)|^2)\Big)\\
	&\le \cdots
	\le  C\tau(1+|\log(\epsilon)|^2),
	\end{align*}
	which completes the proof. 
\end{proof}

With slight modification of our approach, one can obtain the same strong convergence rate for the exponential Euler method or the accelerated exponential Euler method for RSlogS equations. 
Combining the approximation error between Eq.  \eqref{SlogS} and Eq. \eqref{Reg-SlogS} in Lemma \ref{lm-con}, and the strong convergence result in Propositions  \ref{tm-con-add}-\ref{tm-con-real}, we obtain the following convergence result.

\begin{tm}\label{tm-con-total}
	Let Assumption \ref{main-reg-fun} and the condition of Theorem \ref{mild-general} hold.  Assume in addition that $f_{\epsilon}$ satisfies \eqref{con-f}.
	Then the numerical solution of \eqref{spl-add}-\eqref{spl-mul-rea} is strongly convergent to the exact one of Eq. \eqref{SlogS}. Moreover, for $p\ge2$ and  $\delta\in \big(0, \max(\frac 2{\max(d-2,0)},1)\big),$     there exist $C(Q,T,\lambda,p,u_0,\delta,\widetilde g)>0$ and $C(Q,T,\lambda,p,u_0,\alpha,\delta, $ $\widetilde g) >0$ such that when $\mathcal O$ is a bounded domain,
	\begin{align*}
	\sup_{k\le N}\|u^{\epsilon}_k-u(t_k)\|_{L^{p}(\Omega;\mathbb H)}&\le C(Q,T,\lambda,p,u_0,\delta,\widetilde g) ((1+|\log(\epsilon)|)\tau^{\frac 12}+\epsilon^{\frac 12}+\epsilon^{\frac {{\delta}}2}),
	\end{align*}
	and when $\mathcal O=\mathbb R^d,$
	\begin{align*}
	\sup_{k\le N}\|u^{\epsilon}_k-u(t_k)\|_{L^{p}(\Omega;\mathbb H)}&\le C(Q,T,\lambda,p,u_0,\alpha,\delta,\widetilde g) ((1+|\log(\epsilon)|)\tau^{\frac 12}+\epsilon^{\frac {\alpha}{2\alpha+d}}+\epsilon^{\frac {{\delta}}2}).
	\end{align*}
\end{tm}

\begin{cor}\label{cor-entropy-scheme}
	Let the condition of Theorem \ref{tm-con-total} hold.
	Then the regularized  entropy of \eqref{spl-add}-\eqref{spl-mul-rea} is  strongly convergent to the entropy of Eq. \eqref{SlogS}. 
	Furthermore, for $p\ge2$ and $\delta\in \big(0, \max(\frac 2{\max(d-2,0)},1)\big),$   there exist $C(Q,T,\lambda,p,u_0,\delta,\widetilde g)>0$ and $C(Q,T,\lambda,p,u_0,\alpha,\delta,\widetilde g)>0$  such that when $\mathcal O$ is a bounded domain,
	\begin{align*}
	\|F(|u(t_k)|^2)-F(|u_k|^2)\|_{L^p(\Omega)} \le C(Q,T,\lambda,p,u_0,\delta,\widetilde g)((1+|\log(\epsilon)|^2)\tau^{\frac 12}+\epsilon^{\frac 12}+\epsilon^{\frac {{\delta}}2}),
	\end{align*}
	and when $\mathcal O=\mathbb R^d,$ {\small 
		\begin{align*}
		\|F(|u(t_k)|^2)-F(|u_k|^2)\|_{L^p(\Omega)}
		&\le C(Q,T,\lambda,p,u_0,\alpha,\delta,\widetilde g)((1+|\log(\epsilon)|^2)\tau^{\frac 12}+\epsilon^{\frac {\alpha}{2\alpha+d}}+\epsilon^{\frac {{\delta}}2}).
		\end{align*}}
\end{cor}

We have known that $H(u^{\epsilon})$ is an approximation of the original energy $H(u)$ in Proposition \ref{reg-ene}. One may expect that the splitting regularized scheme is also an  approximation of the energy functional. However, the analysis is even more intricate than expected. Some new techniques are needed to get the  convergence of the energy for splitting scheme due to loss of the regularity in time and space of the mild solution.
This will be studied in the future.

\section{Structure-preserving regularized finite difference type splitting scheme}
In the section, we will propose several regularized finite difference schemes, including the regularized Crank--Nilcoson scheme and the regularized mid-point scheme, to study the error between the regularized energy and original one. Throughout this section, we assume that there exists a small $\tau_0(T,\lambda,g,Q,u_0,\epsilon)>0$ such that for $\tau<\tau_0$, there exists a numerical solution for the proposed scheme. Indeed, $\tau_0$ will be depending on $\frac 1{\log(|\epsilon|)}$ according to $(A1)$.
Let $\tau<\tau_0$ be the time step size such that $T=N\tau$.  

\subsection{Regularized mid-point scheme}
We present the framework analyzing the properties of finite difference methods for SlogS equations in terms of  the regularized mid-point scheme. This scheme reads
\begin{equation}
\begin{split}
\label{mid-point}
\Phi_{S,\mathcal F_{t_k}}^{\tau}(u_{k}^{\epsilon})&= u_{k}^{\epsilon}+\int_{t_{k}}^{t_{k+1}}\widetilde g(\Phi_{S,\mathcal F_{t_k}}^{s}( u_{k}^{\epsilon}))\star dW(s),\\
u_{k+1}^{\epsilon}&=\Phi_{\Delta+f}^{\tau}(\Phi_{S,\mathcal F_{t_k}}^{\tau}(u_{k}^{\epsilon})):=\Phi_{S,\mathcal F_{t_k}}^{\tau}(u_{k}^{\epsilon})+\mathbf i \Delta \frac {\Phi_{S,\mathcal F_{t_k}}^{\tau}(u_{k}^{\epsilon})+ u_{k+1}^{\epsilon}}2\tau\\
&\quad+\mathbf i \lambda \tau f_{\epsilon}(|\frac { \Phi_{S,\mathcal F_{t_k}}^{\tau}(u_{k}^{\epsilon})+ u_{k+1}^{\epsilon}}2|^2)\frac {\Phi_{S,\mathcal F_{t_k}}^{\tau}(u_{k}^{\epsilon})+u_{k+1}^{\epsilon}}2,
\end{split}
\end{equation}
where $u_k^{\epsilon}$, $k\le N,$ is the numerical solution at $k$th step and $u_0^{\epsilon}=u_0.$

It can be verified that $\Phi_{S,\mathcal F_{t_k}}$ has the analytic solution if one of the following cases holds: ${\rm Case~1.}$ $W(t)\in\mathbb H$ and $\widetilde g=1;$ ${\rm Case~2.}$ $W(t)\in\mathbb H$ and $\widetilde g=\mathbf i x;$ ${\rm Case~3.}$ ${W(t)}\in L^2(\mathcal O;\mathbb R)$ and $\widetilde g(x)=\mathbf ig(|x|^2)x.$
In these three cases, \eqref{mid-point} becomes a numerical scheme. 
Otherwise, some numerical solver $\widetilde \Phi_{S,\mathcal F_{t_k}}^t$ is needed to discretize $\Phi_{S,\mathcal F_{t_k}}^t.$ For example, one may use the Euler method and get 
\begin{align*}
\widetilde \Phi_{S,\mathcal F_{t_k}}^t(u_{k}^{\epsilon}):
=& u_{k}^{\epsilon}-\frac 12 \sum_{i\in \mathbb N^+} |Q^{\frac 12}e_i|^2|g(| u_{k}^{\epsilon})|^2)|^2 u_{k}^{\epsilon} (t-t_k)+\mathbf i \int_{t_{k}}^{t} g(| u_{k}^{\epsilon}|^2)u_{k}^{\epsilon} dW(s)\\
&-\bi \sum_{i\in \mathbb N^+} Im(Q^{\frac 12}e_i)Q^{\frac 12}e_i (t-t_k)g'(|u_{k}^{\epsilon}|^2) g(| u_{k}^{\epsilon}|^2)|u_{k}^{\epsilon}|^2 u_{k}^{\epsilon}.
\end{align*}
For simplicity, let us deal with the case that $\Phi_{S,\mathcal F_{t_k}}^t$ has an analytic solution since the numerical analysis of other discrete scheme with the numerical solver $\widetilde \Phi_{S,\mathcal F_{t_k}}^t$ is similar.

First, we would like to present the structure-preserving properties, including the symplectic structure and the mass evolution law, of \eqref{mid-point}, which is summarized as follows.

\begin{prop}
	Let Assumption \ref{main-reg-fun} and the condition of Theorem \ref{mild-general} hold. Assume that $\{W(t)\}_{t\ge 0}$ is $\mathbb H$-valued, $\widetilde g=1$ or that  $\{W(t)\}_{t\ge 0}$ is $L^2(\mathcal O;\mathbb R)$-valued, $\widetilde g(x)=\mathbf ig(|x|^2)x.$ Then the phase flow of  \eqref{mid-point} preserves the stochastic symplectic structure, i.e., $\bar \omega_{k+1}=\bar \omega_k$. 
	
	Assume that $\widetilde g(x)=\mathbf i x$ or $\widetilde g(x)=1$. Then \eqref{mid-point} preserves the evolution law of the mass of the exact solution, i.e.,
	\begin{align*}
	\E\Big[M(u_{k+1}^{\epsilon})\Big]=\E\Big[M(u_{k}^{\epsilon})\Big]+\tau \sum_{i\in \mathbb N^+}\|Q^{\frac 12}e_i\|^2 \chi_{\{\widetilde g=1\}},
	\end{align*}
	where $\chi_{\{\widetilde g=1\}}=1$ for the additive noise case and $\chi_{\{\widetilde g=1\}}=0$ for the multiplicative noise case.
\end{prop}

Due to the discretization of $e^{\mathbf i \Delta \tau}$ and loss of regularity in both time and space, the {strong} convergence order of \eqref{Crank} is less than that of splitting schemes \eqref{spl-add}-\eqref{spl-mul-rea}.

\begin{prop}\label{tm-con-mid}
	Let Assumption \ref{main-reg-fun} and the condition of Theorem \ref{mild-general} hold.  Assume in addition that $f_{\epsilon}$ satisfies \eqref{con-f}.
	Then the numerical solution of \eqref{mid-point} is strongly convergent to the exact one of Eq. \eqref{Reg-SlogS}. Moreover, for $p\ge2,$  there exists $C(Q,T,\lambda,p,u_0,\widetilde g)>0$ such that
	\begin{align*}
	\sup_{k\le N}\|u^{\epsilon}_k-u^{\epsilon}(t_k)\|_{L^{p}(\Omega;\mathbb H)}&\le C(Q,T,\lambda,p,u_0,\widetilde g)\epsilon^{-1}((1+|\log(\epsilon)|)\tau^{\frac 12}+\tau^{\frac 13}).
	\end{align*}
\end{prop}

\begin{proof}
	We only give the details for the multiplicative noise case since the proof of the additive noise case is analogous.  
	First, we show the uniform boundedness in $\mathbb H^1$ of $u^{\epsilon}_{k+1}$. 
	Multiplying $ u^{\epsilon}_{k+\frac 12}=\frac {\Phi_{S,\mathcal F_{t_k}}^{\tau}(u_{k}^{\epsilon})+ u^{\epsilon}_{k+1}}2$ on the second equation of $ u^{\epsilon}_{k+1}$ in \eqref{mid-point}, using integration by parts and the fact that  $f_{\epsilon}$ is real-valued, we have that $
	\|u^{\epsilon}_{k+1}\|^2=\|\Phi_{S,\mathcal F_{t_k}}^{\tau}(u_{k}^{\epsilon})\|^2.$ Then from the Burkholder inequality, the growth condition on $g$ and Gronwall's inequality, it follows that 
	\begin{align*}
	\|u^{\epsilon}_{k+1}\|_{L^p(\Omega;\mathbb H)}&\le \exp(C\tau) \|\Phi_{S,\mathcal F_{t_k}}^{\tau}(u_{k}^{\epsilon})\|_{L^p(\Omega;\mathbb H)}\le  \exp(C\tau) \|u^{\epsilon}_{k}\|_{L^p(\Omega;\mathbb H)}.
	\end{align*}
	Similarly, multiplying $\Delta  u^{\epsilon}_{k+\frac 12}$ on the equation of $u^{\epsilon}_{k+1}$ in \eqref{mid-point}, repeating the above procedures and using (A5), we get that
	\begin{align*}
	\|\nabla u^{\epsilon}_{k+1}\|^2&\le \|\nabla \Phi_{S,\mathcal F_{t_k}}^{\tau}(u_{k}^{\epsilon})\|^2+C\tau\|\nabla u^{\epsilon}_{k+\frac 12}\|^2
	\le \frac {1+ C\tau}{1-C\tau} \|\nabla \Phi_{S,\mathcal F_{t_k}}^{\tau}(u_{k}^{\epsilon})\|^2\\
	& \le \exp(C\tau) \|\nabla \Phi_{S,\mathcal F_{t_k}}^{\tau}(u_{k}^{\epsilon})\|^2,\\
	\| \Phi_{S,\mathcal F_{t_k}}^{\tau}(u_{k}^{\epsilon})\|_{L^p(\Omega;\mathbb H^1)}&\le \exp(C\tau) \| u^{\epsilon}_{k}\|_{L^p(\Omega;\mathbb H^1)}.
	\end{align*} 
	Combining the above estimates, we conclude that 
	\begin{align*}
	\sup_{k\le N}\|u^{\epsilon}_{k}\|_{L^p(\Omega;\mathbb H^1)}+
	\sup_{k\le N-1}\sup_{t\in [0,\tau]}\|\Phi_{S,\mathcal F_{t_k}}^{t}(u_{k}^{\epsilon})\|_{L^p(\Omega;\mathbb H^1)}&\le C(Q,T,\lambda,p,u_0,\widetilde g).
	\end{align*}
	Now we are in a position to present the error estimate of \eqref{mid-point}.
	We introduce the following auxiliary process $\widehat u$,
	\begin{align*}
	\widehat u(t)&= \Phi_{S,\mathcal F_{t_k}}^{t-t_k}\prod_{i=0}^{k-1}(\Phi_{\Delta+f}^{\tau} \Phi_{{S,\mathcal F_{t_i}}}^{\tau}) u_0=\Phi_{S,\mathcal F_{t_k}}^{t-t_k} u_{k}^{\epsilon}, \;\text{if}\; t\in [t_k,t_{k+1}), \; k\le N-1,\\
	\widehat u(t_{k+1})&= \Phi_{\Delta+f}^{\tau}\lim_{t\to t_{k+1}} \widehat u(t)=u^{\epsilon}_{k+1},\; \widehat u(0)=u_0.
	\end{align*} 
	According to the definition, it is not hard to see that $\widehat u$ is right-continuous with left limit, and thus a predictable process.
	Then the mild form of $\widehat u(t)$ is 
	\begin{align*}
	\widehat u(t)
	=&S_{\tau}^k u_0 +\int_0^tS_{k,t}(s)\Big(-\frac 12 (g(|\widehat u(s)|^2))^2\widehat u(s)\sum_{i}|Q^{\frac 12}e_i|^2 \\
	&- \bi g(|\widehat u(s)|^2)g'(|\widehat u(s)|^2)|\widehat u(s)|^2\widehat u(s) \sum_{i} Im(Q^{\frac 12}e_i) Q^{\frac 12}e_i\Big)ds\\
	&+\int_0^tS_{k,t}(s)\bi g\big(|\widehat u(s)|^2)\widehat u(s)dW(s)+\mathbf i \lambda \int_{0}^{t_k}S_{k,t}(s)T_{\tau} f_{\epsilon}(|\widehat u_{[s]+\frac 12}|^2)\widehat u_{[s]+\frac 12}ds,
	\end{align*}
	where $S_{\tau}=\frac {I+\frac 12\bi \tau\Delta }{I-\frac 12\bi \tau \Delta}$, $T_{\tau}=\frac 1{I-\frac 12\bi \tau \Delta}$ and $S_{k,t}(s)=\sum\limits_{j=1}^{k}\chi_{[t_{k-j},t_{k+1-j}]}(s) S_{\tau}^j(s) +\chi_{[t_{k,}t]}(s)$, $[s]$ is integer part of $\frac {s}{\tau}$ and $\widehat u_{[s]+\frac 12}=\frac {\Phi_{S,\mathcal F_{[s]\tau}}(\widehat u([s]\tau)+\widehat u(([s]+1)\tau)} 2.$
	Now, we get {\small
		\begin{align*}
		u^{\epsilon}(t)-\widehat u(t)=&(e^{\mathbf i \Delta t}-S_{\tau}^{k})u_0+\bi \lambda  \int_0^{t_k} S_{k,t}(s)T_{\tau}\Big(f_{\epsilon}(|u^{\epsilon}(s)|^2) u^{\epsilon}(s)- f_{\epsilon}(|\widehat u_{[s]+\frac 12}|^2)\widehat u_{[s]+\frac 12}\Big)ds\\
		&-\frac 12 \int_0^tS_{k,t}(s)(II_{mod}(u^{\epsilon}(s))-II_{mod}(\widehat u(s)))ds\\
		&-\int_0^t (e^{\mathbf i\Delta (t-s)}-S_{k,t}(s))II_{mod}(u^{\epsilon}{(s)})ds\\
		&+\bi \lambda  \int_0^t (e^{\mathbf i\Delta (t-s)}-S_{k,t}(s) T_{\tau})f_{\epsilon}(|u^{\epsilon}(s)|^2) u^{\epsilon}(s)ds\\
		&+\bi\int_0^tS_{k,t}(s)\Big(g(|u^{\epsilon}(s)|^2)u^{\epsilon}(s)-g(|\widehat u(s)|^2)\widehat u(s)\Big)dW(s)\\
		&+\bi\int_0^t(e^{\mathbf i\Delta (t-s)}-S_{k,t}(s))\Big(g(|u^{\epsilon}(s)|^2)u^{\epsilon}(s)\Big)dW(s)\\
		=:&II_1+II_2+II_3+II_4+II_5+II_6+II_7.
		\end{align*}
		B}y means of Fourier transform and Plancherel's equality (see e.g. \cite[Appendix]{BD06}), it holds that 
	\begin{align*}
	\|e^{\mathbf i k\Delta{\tau}}-S_{\tau}^{k}\|_{\mathcal L(\mathbb H^1,\mathbb H)}&\le C(T)\tau^{\frac 13},\; \|S_{\tau}-I\|_{\mathcal L(\mathbb H^1,\mathbb H)}\le C\tau^{\frac 12},\;
	\|T_{\tau}-I\|_{\mathcal L(\mathbb H^1,\mathbb H)}\le C\tau^{\frac 12}.
	\end{align*}
	Thus, we have 
	$\|II_1\|\le C\tau^{\frac 13}\|u_0\|_{\mathbb H^1}.$
	To bound $II_2$, we adopt the continuity estimate of $\widehat u(s)$ on each small interval. 
	More precisely, for $s\in[t_{j-1},t_j]$, $j\le k$, {\small
		\begin{align*}
		&\|\widehat u(s)-\widehat u_{[s]+\frac 12}\|
		\le \frac 12\|\widehat u(s)-\Phi_{S,\mathcal F_{j-1}}^{\tau}(u_{j-1}^{\epsilon})\|+\frac 12\|\widehat u(s)-u_{j}^{\epsilon}\|\\
		\le& \frac 12\|\int_{t_{j-1}}^sII_{mod}(\widehat u(s))ds\|+\frac 12\|\int_{t_{j-1}}^s g(|\widehat u(s)|^2) \widehat u(s)dW(s)\|+\frac 12\tau|\lambda|\|f_{\epsilon}(|\widehat u_{j-\frac 12}|^2)\widehat u_{j-\frac 12}\|\\
		&+ \frac 12\|(S_{\tau}-I)u_{j-1}^{\epsilon}\|+\frac 12\|\int_{s}^{t_j}II_{mod}(\widehat u(s))ds\|+\frac 12\|\int_{s}^{t_j} g(|\widehat u(s)|^2) \widehat u(s)dW(s)\|.
		\end{align*}
		T}aking $p$th moment, using the growth condition of $f$ and $g$, as well as a priori estimate of $\widehat u$ in $\mathbb H^1$,
	we obtain 
	$
	\|\widehat u(s)-\widehat u_{[s]+\frac 12}\|_{L^p(\Omega;\mathbb H)}\le C\tau^{\frac 12}+C\log(|\epsilon|)\tau.
	$
	Then the boundedness of $S_{k,t}$ and the property \eqref{con-f} of $f_{\epsilon}$ lead to 
	\begin{align*}
	&\|II_2\|_{L^p(\Omega;\mathbb H)}\\
	&\le  C\int_0^t \Big\|f_{\epsilon}(|u^{\epsilon}(s)|^2) u^{\epsilon}(s)-f_{\epsilon}(|\widehat u(s)|^2)\widehat u(s)\Big\|_{L^p(\Omega;\mathbb H)}ds+C(\tau^{\frac 12}+\log(|\epsilon|)\tau)\\
	&\le C\int_0^t (1+|\log(\epsilon)|)\|u^{\epsilon}(s)-\widehat u(s)\|_{L^p(\Omega;\mathbb H)}ds+C(\tau^{\frac 12}+\log(|\epsilon|)\tau).
	\end{align*}
	Based on the assumption on $g,$ we arrive at 
	\begin{align*}
	\|II_3\|
	&\le C\sum_{i}\|Q^{\frac 12}e_i\|_{L^{\infty}}^2\int_0^t\|u^{\epsilon}(s)-\widehat u(s)\|ds.
	\end{align*}
	By using the continuity of $e^{\bi\Delta}$ and Lemma \ref{lm-con}, we have that 
	\begin{align*}
	\|II_4\|+\|II_5\|
	&\le C\tau^{\frac 13}\int_0^t (1+|\log(\epsilon)|)\|u^{\epsilon}\|_{\mathbb H^1}ds\\
	&\quad+C\tau^{\frac 13}\int_0^t \Big(\sum_{i}\|Q^{\frac 12}e_i\|_{L^{\infty}}^2+\sum_{i}\|\nabla Q^{\frac 12}e_i\|_{L^{\infty}}^2\Big)\|u^{\epsilon}\|_{\mathbb H^1}ds.
	\end{align*}
	Applying the Burkerholder inequality and the growth condition of $g$, we obtain that 
	\begin{align*}
	\E\Big[\sup_{s\in[0,t]}\|II_6\|^p\Big]
	&\le C\E\Big[\Big(\int_0^t\sum_{i}\Big\|(g(|u^{\epsilon}(s)|^2))u^{\epsilon}(s)-g(|\widehat u(s)|^2)\widehat u(s)) Q^{\frac 12}e_i\Big\|^2ds\Big)^{\frac p2}\Big]\\
	&\le C(\sum_{i}\|Q^{\frac 12}e_i\|^2_{L^{\infty}})^{\frac p2}\int_0^t\E\Big[\|u^{\epsilon}(s)-\widehat u(s)\|^p\Big]ds.
	\end{align*}
	and 
	\begin{align*}
	\E\Big[\sup_{s\in[0,t]}\|II_7\|^p\Big]
	&\le \tau^{\frac p3}C(\sum_{i}\|Q^{\frac 12}e_i\|^2_{L^{\infty}}+\|\nabla Q^{\frac 12}e_i\|^2_{L^{\infty}})^{\frac p2}\int_0^t\E\Big[\|u^{\epsilon}\|_{\mathbb H^1}^p\Big]ds.
	\end{align*}
	Now taking $p\ge 2$, we get 
	$
	\E\Big[\sup\limits_{s\in[0,t]}\|u^{\epsilon}(s)-\widehat u(s)\|^p\Big]
	\le C\sum_{j=1}^7\E\Big[\sup\limits_{s\in[0,t]}\|II_j\|^p\Big].
	$
	We complete the proof by combining the above estimates.
\end{proof}

Combining  Lemma \ref{lm-con} with Proposition \ref{tm-con-mid}, we obtain the following theorem.

\begin{tm}\label{cor-con-total}
	Let the condition of Proposition \ref{tm-con-mid} hold.
	Then the numerical solution of \eqref{mid-point} is strongly convergent to the exact solution of Eq. \eqref{SlogS}. Moreover, for $p\ge2$ and $\delta\in \big(0, \max(\frac 2{\max(d-2,0)},1)\big),$  there exist $C(Q,T,\lambda,p,u_0,\widetilde g)>0$ and $C(Q,T,\lambda,p,u_0,\alpha,\widetilde g)>0$ such that when $\mathcal O$ is a bounded domain,
	\begin{align*}
	{\sup_{k\le N}}\|u^{\epsilon}_k-u(t_k)\|_{L^{p}(\Omega;\mathbb H)}&\le  C(Q,T,\lambda,p,u_0,\widetilde g)( \epsilon^{-1}((1+|\log(\epsilon)|)\tau^{\frac 12}+\tau^{\frac 13})+\epsilon^{\frac 12}+\epsilon^{\frac {{ \delta}}2}),
	\end{align*}
	and when $\mathcal O=\mathbb R^d,$
	\begin{align*}
	{\sup_{k\le N}}\|u^{\epsilon}_k-u(t_k)\|_{L^{p}(\Omega;\mathbb H)}&\le  C(Q,T,\lambda,p,u_0,\alpha,\widetilde g)(\epsilon^{-1}((1+|\log(\epsilon)|)\tau^{\frac 12}+\tau^{\frac 13})+\epsilon^{\frac {\alpha}{2\alpha+d}}+\epsilon^{\frac {{ \delta}}2}).
	\end{align*}
\end{tm}

\subsection{Regularized Crank--Nicolson scheme}
Based on the study of the regularized mid-point scheme,{ we are in a position to study the properties of the regularized Crank--Nicolson scheme and show that it is a good approximation of the energy of Eq. \eqref{SlogS}. }
The regularized splitting Crank--Nicolson type scheme reads 
\begin{align}\label{Crank}
\Phi_{S,\mathcal F_{t_k}}^{\tau}(u_{k}^{\epsilon})&:= u_{k}^{\epsilon}+\int_{t_{k}}^{t_{k+1}}\widetilde g(\Phi_{S,\mathcal F_{t_k}}^{s}( u_{k}^{\epsilon}))\star d W(s), \\\nonumber 
u_{k+1}^{\epsilon}&=\Phi_{\Delta+f}^{\tau}(\Phi_{S,\mathcal F_{t_k}}^{\tau}(u_{k}^{\epsilon})):=\Phi_{S,\mathcal F_{t_k}}^{\tau}(u_{k}^{\epsilon})+\mathbf i \Delta \frac {\Phi_{S,\mathcal F_{t_k}}^{\tau}(u_{k}^{\epsilon})+ u_{k+1}^{\epsilon}}2\tau\\\nonumber 
&+\mathbf i \lambda \tau\int_0^1 f_{\epsilon}(\theta|\Phi_{S,\mathcal F_{t_k}}^{\tau} (u_{k}^{\epsilon})|^2+(1-\theta)| u_{k+1}^{\epsilon}|^2)d\theta  \frac {\Phi_{S,\mathcal F_{t_k}}^{\tau}(u_{k}^{\epsilon})+u_{k+1}^{\epsilon}}2,\nonumber
\end{align}
where $u_k^{\epsilon}$, $k\le N,$ is the numerical solution at the $k$th step and $u_0^{\epsilon}=u_0.$
By the chain rule, one can verify that
\begin{align*}
&\int_0^1 f_{\epsilon}(\theta|\Phi_{S,\mathcal F_{t_k}}^{\tau} (u_{k}^{\epsilon})|^2+(1-\theta)| u_{k+1}^{\epsilon}|^2) d\theta \frac {\Phi_{S,\mathcal F_{t_k}}^{\tau}(u_{k}^{\epsilon})+u_{k+1}^{\epsilon}}2\\
=&\frac {\widetilde F_{\epsilon}(| u_{k+1}^{\epsilon}|^2)-\widetilde  F_{\epsilon}(| \Phi_{S,\mathcal F_{t_k}}^{\tau}(u_{k}^{\epsilon})|^2)}{|u_{k+1}^{\epsilon}|^2-|\Phi_{S,\mathcal F_{t_k}}^{\tau}(u_{k}^{\epsilon}) |^2} \frac {\Phi_{S,\mathcal F_{t_k}}^{\tau}(u_{k}^{\epsilon})+ u_{k+1}^{\epsilon}}2,
\end{align*}
where $\widetilde F_{\epsilon}$ is the integrand in the regularized  entropy $F_{\epsilon}.$

In the following, we focus on the uniform a priori estimate of \eqref{Crank} when $\widetilde g(x)=\mathbf i x$ and $\mathcal O$ is bounded.  For convenience, assume that $u_0$ is a deterministic function. We will study the convergence of \eqref{Crank} for general diffusion coefficient case in the future.

\begin{prop}\label{prop-con-crank}
	Let the condition of  Proposition \ref{tm-con-mid} hold,  $\widetilde g(x)=\mathbf i x$, $\mathcal O$ be a bounded domain. 
	Then the modified energy of the numerical solution of the Crank--Nicolson type method \eqref{Crank} is well-defined. 
	Moreover, for $p\ge2,$ there exists $C(Q,T,\lambda,$ $p,u_0)>0$ such that
	\begin{align*}
	\sup_{k\le N}\|\nabla u^{\epsilon}_k\|_{L^{p}(\Omega;\mathbb H)}&\le C(Q,T,\lambda,p,u_0).
	\end{align*}
\end{prop}

\begin{proof} 
	The proof is similar to that of Proposition \ref{tm-con-mid}.
	The main difference lies on the a priori estimate of the auxiliary process $\widehat u$ defined by 
	\begin{align*}
	\widehat u(t)&= \Phi_{S,\mathcal F_{t_k}}^{t-t_k}{\prod_{i=0}^{k-1}(\widehat \Phi_{\Delta+f}^{\tau} \Phi_{S,\mathcal F_{t_i}}^{\tau})}
	u_0, \;\text{if}\; t\in [t_k,t_{k+1}), \; k\le N-1,\\
	\widehat u(t_{k+1})&= \widehat \Phi_{\Delta+f}^{\tau}\lim_{t\to t_{k+1}} \widehat u(t)=u^{\epsilon}_{k+1},\quad
	\widehat u(0)=u_0.
	\end{align*} 
	Following the same steps as in the proof of Proposition \ref{tm-con-mid}, we obtain
	\begin{align}\label{ene-pri}
	{\sup_{k\le N-1}}\|\sup_{t\in [0,\tau]} H_{\epsilon}(\Phi_{S,\mathcal F_{t_k}}^{t}(u_{k}^{\epsilon}))\|_{L^p(\Omega;\mathbb H)}\le C(Q,T,\lambda,p,u_0).
	\end{align}
	By the Gagliardo--Nirenberg interpolation inequality \eqref{gn-sob}, the procedures in the proof of Lemma \ref{lm-mod-en}, Young's and H\"older's inequalities, it  can be verified that for a small enough ${ \delta'}>0, \eta>0$ and $\delta<\frac {2}d,$
	\begin{align*}
	\frac 12\|\nabla v\|^2
	&\le  |H_{\epsilon}(v)|+{ \delta'} \|\nabla v\|^2+C({\delta'})(1+\|v\|^2+\|v\|^{\frac {4+4\delta-2\delta d}{2-\delta d}})
	+\|v\|_{L^{2-2\eta}}^{2-2\eta}.
	\end{align*}
	Since $\mathcal O$ is bounded, by \eqref{ene-pri}, we obtain  
	$${\sup\limits_{k\le N-1}}\sup\limits_{t\in [0,\tau]} \|\Phi_{S,\mathcal F_{t_k}}^{t}(u_{k}^{\epsilon})\|_{L^p(\Omega;\mathbb H^1)}\le C(u_0,\lambda,T,Q,{p}),$$ 
	which completes the proof.
\end{proof}

Compared to \eqref{mid-point}, \eqref{Crank} fails to preserve the stochastic symplectic structure.
However, it preserves the mass evolution law of the exact solution.

\begin{prop}
	Let the condition of  Theorem \ref{tm-con-det} hold and $\widetilde g(x)=\mathbf i x$ or $\widetilde g(x)=1$. 
	Then \eqref{Crank} preserves the evolution law of the mass of the exact solution, that is, 
	\begin{align*}
	\E\Big[M(u_{n+1}^{\epsilon})\Big]=\E\Big[M(u_{n}^{\epsilon})\Big]+\tau \sum_{i \in \mathbb N^+}\|Q^{\frac 12}e_i\|^2 \chi_{\{\widetilde g=1\}},
	\end{align*}
	where $\chi_{\{\widetilde g=1\}}=1$ for the additive noise case and $\chi_{\{\widetilde g=1\}}=0$ for the multiplicative noise case.
\end{prop}

In the following, we give the error estimate  of \eqref{Crank} by making use of  \eqref{mid-point} and solve its convergence problem in temporal direction.
The convergence analysis of \eqref{Crank} is  more complicated than  that of the splitting type scheme since the boundedness of energy of numerical solution may not imply the boundedness of the numerical solution under $\mathbb H^1$-norm.
Generally speaking, the a priori estimate of  \eqref{Crank} may be not uniform with respect to ${\epsilon}$.

\begin{tm}\label{tm-con-det}
	Let $\widetilde g=0, 1$ or ${\bi} x$, and $\mathcal O$ be bounded.  
	Assume that $f_{\epsilon}(x)=\log(\frac {\epsilon+x}{1+\epsilon x})$, $x>0$.
	Then the numerical solution of \eqref{Crank} is strongly convergent to the exact one of Eq. \eqref{SlogS}. Moreover, for $p\ge1$ and $\delta\in \big(0, \max(\frac 2{\max(d-2,0)},1)\big),$   there exists $C(Q,T,\lambda,u_0,\delta,\widetilde g,p)>0$  such that
	\begin{align*}
	&\sup_{k\le N}\|u^{\epsilon}_k-u(t_k)\|_{L^p(\Omega;\mathbb H)}\\\le& C(Q,T,\lambda,u_0,\delta,\widetilde g,p)\Big((\epsilon^{-1}(1+|\log(\epsilon)|)\tau^{\frac 12}+\tau^{\frac 13})+\epsilon^{\frac 12}+\epsilon^{\frac {{\delta}}2}
+ \epsilon^{-1}\tau^{\frac 14}|\log(\epsilon)|^{\frac 12}\Big).
	\end{align*}
	
\end{tm}

\begin{proof}
	Thanks to  Proposition \ref{prop-con-crank} and the continuity of $e^{\bi \Delta \tau}$,  it holds that 
	$
	\|u^{\epsilon}_{k+1}-\Phi^\tau_S(u^{\epsilon}_k)\|_{L^p(\Omega;\mathbb H)}\le C(\tau^{\frac 12}+\tau |\log(\epsilon)|).
	$
	By expanding the flow $\widetilde \Phi_{\Delta+f}$ of \eqref{Crank} via $ \Phi_{\Delta+f}$ of \eqref{mid-point}, and following the same procedures {as in} the proof of Theorem \ref{tm-con-mid}, we obtain
	\begin{align*}
	u^{\epsilon}(t)-\widehat u(t):=II_1+II_2+II_3+II_4+II_5+II_6+II_7+II_8,
	\end{align*}
	where $II_1-II_7$ are presented in the proof of Theorem \ref{tm-con-mid}.
	Here $II_8$ is the expansion error $\widetilde \Phi_{\Delta+f}$ given by 
	\begin{align*}
	II_8 :=&\bi \lambda  \int_0^{t_k} S_{k,t}(s)T_{\tau}\Big(f_{\epsilon}(|\widehat u_{[s]+\frac 12}|^2)\nonumber\\
	&-\int_0^1f_{\epsilon}(\theta |\Phi^\tau_{S,\mathcal F_{t_k}} (u^{\epsilon}_k)|^2+(1-\theta)|u^{\epsilon}_{k+1}|^2) d\theta\Big)\widehat u_{[s]+\frac 12}\Big)ds.
	\end{align*}
	In order to estimate $II_8$, let us consider the event that $|\Phi^\tau_{S,\mathcal F_{t_k}} (u^{\epsilon}_k)|^2\ge |u^{\epsilon}_{k+1}|^2$.
			The estimate on the event that $|\Phi^\tau_{S,\mathcal F_{t_k}} (u^{\epsilon}_k)|^2\le |u^{\epsilon}_{k+1}|^2$ is similar.
	
	The convexity of $|\cdot|^2$ implies that 
	$|\widehat u^{\epsilon}_{k+\frac 12}|^2\le \frac 12|\Phi^\tau_{S,\mathcal F_{t_k}} (u^{\epsilon}_k)|^2 +\frac 12|u^{\epsilon}_{k+1}|^2.$
	Assume that  $\theta_0\in (0,\frac 12)$ is the largest  number such that $|\widehat u^{\epsilon}_{k+\frac 12}|^2= \theta_0 |\Phi^\tau_{S,\mathcal F_{t_k}} (u^{\epsilon}_k)|^2+(1-\theta_0)|u^{\epsilon}_{k+1}|^2$.  
	Otherwise, the proof of the desired estimate is simple by choosing one of the following estimates.  
	Then when $\theta\ge \theta_0$ it holds that for $\delta'_1\in (0,1),$
	\begin{align*}
	&\Big|f_{\epsilon}(|\widehat u_{k+\frac 12}|^2)-f_{\epsilon}(\theta |\Phi^\tau_{S,\mathcal F_{t_k}} (u^{\epsilon}_k)|^2+(1-\theta)| u^{\epsilon}_{k+1}|^2)\Big|^{\delta'_1}\\
	\le& C\Big(\frac {\theta |\Phi^\tau_{S,\mathcal F_{t_k}} (u^{\epsilon}_k)|^2+(1-\theta)| u^{\epsilon}_{k+1}|^2-|\widehat u^{\epsilon}_{k+\frac 12}|^2}{\epsilon +|\widehat u_{k+\frac 12}|^2}\Big)^{\delta'_1}\\
	&+C\epsilon^{\delta'_1} \Big(\frac { \theta |\Phi^\tau_{S,\mathcal F_{t_k}} (u^{\epsilon}_k)|^2+(1-\theta)|u^{\epsilon}_{k+1}|^2- |\widehat u_{k+\frac 12}|^2}{1+\epsilon |\widehat u_{k+\frac 12}|^2}\Big)^{\delta'_1}.
	\end{align*}
	When $\theta\le \theta^0$, the dominant part of $f_{\epsilon}(|\widehat u^{\epsilon}_{k+\frac 12}|^2)-f_{\epsilon}(\theta |\Phi^\tau_{S,\mathcal F_{t_k}} u^{\epsilon}_k|^2+(1-\theta)|u^{\epsilon}_{k+1}|^2)$ is a concave function over $\theta.$ 
	Using Jensen's inequality, we obtain that  {\small
		\begin{align*}
		&\int_0^{\theta^0}(\log(\epsilon +|\widehat u_{k+\frac 12}^{\epsilon}|^2)-\log(\epsilon +\theta |\Phi^\tau_{S,\mathcal F_{t_k}} (u^{\epsilon}_k)|^2+(1-\theta)|u^{\epsilon}_{k+1}|^2)) d\theta\\
		\le& \theta^0 \Big(\log(\epsilon +|\widehat u_{k+\frac 12}^{\epsilon}|^2)-\log(\epsilon +\frac {\theta^0} 2|\Phi_{S,\mathcal F_{t_k}}^\tau (u^{\epsilon}_k)|^2 +(1-\frac {\theta^0}2)|u^{\epsilon}_{k+1}|^2)\Big)\\
		\le& C \theta^0 \frac {|\widehat u_{k+\frac 12}|^2-\frac {\theta^0} 2|\Phi_{S,\mathcal F_{t_k}}^\tau (u^{\epsilon}_k)|^2 -(1-\frac {\theta^0}2)|u^{\epsilon}_{k+1}|^2}{\epsilon +\frac {\theta^0} 2|\Phi_{S,\mathcal F_{t_k}}^\tau (u^{\epsilon}_k)|^2 +(1-\frac {\theta^0}2)|u^{\epsilon}_{k+1}|^2}\\
		\le & C\theta^0 \frac {|\widehat u_{k+\frac 12}|^2-\frac {\theta^0} 2|\Phi_{S,\mathcal F_{t_k}}^\tau (u^{\epsilon}_k)|^2 -(1-\frac {\theta^0}2)|u^{\epsilon}_{k+1}|^2}{\epsilon +\frac {\theta^0} 2|\Phi_{S,\mathcal F_{t_k}}^\tau (u^{\epsilon}_k)|^2 +\frac {\theta^0}2 |u^{\epsilon}_{k+1}|^2}
		\end{align*}
}Combining the above estimates, using the boundedness of $S_{k,t}$ and $T_{\tau}$, we achieve that for $\delta_1'=\frac 12,$
	{\small
		\begin{align*}
		&\|II_8\|_{L^p(\Omega;\mathbb H)}\\
		&\le C\| u^{\epsilon}_{k+1}-\Phi^\tau_{S,\mathcal F_{t_k}} (u^{\epsilon}_k)\|_{L^{2p}(\Omega;\mathbb H)}^{\frac 12}(\| u^{\epsilon}_{k+1}\|_{L^{2p}(\Omega;\mathbb H)}+\|\Phi^\tau_{S,\mathcal F_{t_k}} (u^{\epsilon}_k)\|_{L^{2p}(\Omega;\mathbb H)})^{\frac 12}|\log(\epsilon)|^{\frac 12}\\
		&+C\| u^{\epsilon}_{k+1}-\Phi^\tau_{S,\mathcal F_{t_k}} (u^{\epsilon}_k)\|_{L^{2p}(\Omega;\mathbb H)}.
		\end{align*}
		B}ased on the estimates of $II_1$-$II_7$ in the proof of Proposition \ref{tm-con-mid} and the above estimate of $II_8$, we complete the proof.
\end{proof}

\subsection{Weak convergence of the regularized energy}

In order to show that \eqref{Crank} is a suitable scheme which is convergent to the exact solution in terms of energy, we take  $f_{\epsilon}(x)=\log(\frac {\epsilon+x}{1+\epsilon x})$ to illustrate the main strategy. 
One could generalize the choices of $f_{\epsilon}$ according to the proof of the following proposition.

\begin{prop}\label{mid-eng}
	Let the condition of  Theorem \ref{tm-con-det} hold and $\widetilde g(x)=\mathbf i x$.  Assume that $f_{\epsilon}=\log(\frac {\epsilon+x}{1+\epsilon x})$, $x>0$. Then
	the regularized  energy of \eqref{Crank} is convergent to the energy of  \eqref{SlogS}. Furthermore, for $p\ge1,$ and $\delta\in \big(0, \max(\frac 2{\max(d-2,0)},1)\big),$  there exists $C(Q,T,\lambda,p,u_0)>0$ such that
	\begin{align*} 
	\Big|\E\Big[H_{\epsilon}(u_k^{\epsilon})-H_0(u^0(t_k)\Big] \Big|&\le C(Q,T,\lambda,p,u_0)\Big((\epsilon^{-1}(1+|\log(\epsilon)|)\tau^{\frac 12}+\tau^{\frac 13})+\epsilon^{\frac 12}+\epsilon^{\frac {{\delta}}2}\\
	&\quad+ \epsilon^{-1}\tau^{\frac 14}|\log(\epsilon)|^{\frac 12}\Big).
	\end{align*}
\end{prop}
\begin{proof}
	According to Lemma \ref{reg-ene}, it suffices to estimate $\E[H_{\epsilon}(u^{\epsilon}_k)-H_{\epsilon}(u^{\epsilon}(t_k)].$
	By analyzing the expansion of the energy of $u^{\epsilon}_k$ and $u^{\epsilon}(t_k)$, we obtain
	\begin{align*}
	&\E\Big[H_{\epsilon}(u^{\epsilon}(t_k))-H_{\epsilon}(u^{\epsilon}_k)\Big]\\
	=&\E\Big[H_{\epsilon}(u^{\epsilon}(t_{k-1}))-H_{\epsilon}(\Phi(u^{\epsilon}_{k-1}))\Big]\\
	&+\int_{t_k}^{t_{k+1}}\frac 12\sum_{i}\Big(\|u^{\epsilon}(s)\nabla Q^{\frac 12}e_i\|^2-\|\widehat u(s)\nabla Q^{\frac 12}e_i\|^2\Big)ds\\
	&-\lambda\int_0^t\sum_{i} \E\Big[\Big\<f_{\epsilon}(|u^{\epsilon}|^2)u^{\epsilon},-\frac 12u^{\epsilon} |Q^{\frac 12}e_i|^2\Big\>-\Big\<f_{\epsilon}(|\widehat u|^2)\widehat u,-\frac 12\widehat u |Q^{\frac 12}e_i|^2\Big\>\Big]ds \\
	&-\frac 12 \lambda \int_0^t\sum_{i}\E\Big[\Big\<f_{\epsilon}(|u^{\epsilon}|^2)u^{\epsilon}Q^{\frac 12}e_i,u^{\epsilon}Q^{\frac 12}e_i\Big\>-\Big\<f_{\epsilon}(|\widehat u|^2)\widehat uQ^{\frac 12}e_i,\widehat u Q^{\frac 12}e_i\Big\>\Big]ds\\
	&
	-\frac 12 \lambda \int_0^t\sum_{i}\mathbb E\Big[\Big\<2 Re(\bi \bar u^{\epsilon} u^{\epsilon}Q^{\frac 12}e_i)\frac {\partial f_{\epsilon}}{\partial x}(|u^{\epsilon}|^2)u^{\epsilon},\bi u^{\epsilon}Q^{\frac 12}e_i\Big\>\\
	&\quad-\Big\<2 Re(\bi \bar {\widehat u} \widehat u Q^{\frac 12}e_i)\frac {\partial f_{\epsilon}}{\partial x}(|\widehat u|^2)\widehat u,\bi \widehat uQ^{\frac 12}e_i\Big\>\Big]ds,
	\end{align*}
	where $\widehat u$ is defined in the proof of Proposition \ref{prop-con-crank}.
	By applying the similar estimates in the proof of Lemma \ref{reg-ene} and H\"older's inequality, it follows that for $\delta \le \min(1,\frac {2}{\max(d-2,0)})$ and $\eta\in (0,1),$
	\begin{align*}
	&\quad\Big|\E\Big[H_{\epsilon}(u^{\epsilon}(t_k))-H_{\epsilon}(u^{\epsilon}_k)\Big]\Big|-\Big|\E\Big[H_{\epsilon}(u^{\epsilon}(t_{k-1}))-H_{\epsilon}(u^{\epsilon}_{k-1})\Big]\Big|\\
	&\le
	C\tau \sup_{t\in[t_k,t_{k+1}]}\E\Big[\|u^{\epsilon}(t)-\widehat u(t)\|\|u^{\epsilon}(t)+\widehat u(t)\|\Big]\\
	&\quad + C \tau|\lambda| \sum_{i}\| Q^{\frac 12}e_i\|_{L^{\infty}}^2 \epsilon^{\eta}\Big(1+\sup_{t\in [0,T]}\E\Big[\|u^{\epsilon}(t)\|^{2-2\eta}\Big]+\sup_{t\in [0,T]}\E\Big[\|\widehat u(t)\|^{2-2\eta}\Big]\Big)\\
	&\quad +
	C \tau |\lambda| \sum_{i}\| Q^{\frac 12}e_i\|_{L^{\infty}}^2 \epsilon^{\delta}\sup_{t\in [0,T]}\E\Big[\|u^{\epsilon}(t)\|_{L^{2+2\delta}}^{2+2\delta}+\|\widehat u(t)\|_{L^{2+2\delta}}^{2+2\delta}\Big].
	\end{align*}
	Using Theorem \ref{tm-con-det}, Proposition \ref{reg-ene} and iteration arguments, we finish the proof.
\end{proof}

\section{Appendix}

\begin{prop}\label{prop-evo}
	Let the condition of Theorem \ref{mild-general} and Assumption \ref{main-as} hold.  Let $u^{\epsilon}$ be the mild solution of  Eq. \eqref{Reg-SlogS} and $u^0$ be the mild solution of Eq.  \eqref{SlogS}. When $\widetilde g=1,$ the mild solution $u^{\epsilon}$ is shown to satisfy the following evolution laws, 
	{\small
		\begin{align*}
		M(u^{\epsilon}(t))&=M(u^{\epsilon}_0)+\int_{0}^t\sum_{i\in \mathbb N^+} \|Q^{\frac 12}e_i\|^2ds
		+2\int_{0}^t\<u^{\epsilon}(s),dW(s)\>,\\
		M_{\alpha}(u^{\epsilon}(t))
		&=M_{\alpha}(u^{\epsilon}_0)
		+\int_0^t 4\alpha \<(1+|x|^2)^{\alpha-1}x u^{\epsilon}(s),\bi \nabla u^{\epsilon}(s)\>ds\\
		&\quad+\int_0^t  \sum_{i\in\mathbb N^+}\|Q^{\frac 12}e_i\|_{L^2_{\alpha}}^2ds+ \int_0^t 2\<(1+|x|^2)^{\alpha}u^{\epsilon}(s),dW(s)\>.
		\end{align*}
	}
	{
		When $\widetilde g=\mathbf i g(|x|^2)x$, the mild solution $u^{\epsilon}$ satisfies the following evolution laws,}
	{\small \begin{align*}
		M(u^{\epsilon}(t))&=M(u^{\epsilon}_0)+2\int_{0}^t\<u^{\epsilon}(s),\mathbf  i g(|u^{\epsilon}(s)|^2)u^{\epsilon}(s) dW(s)\>\\
		&\quad+\int_0^t \<u^{\epsilon}(s), -\bi \sum_{k\in \mathbb N^+}Im(Q^{\frac 12}e_i)Q^{\frac 12}e_i g'(|u^{\epsilon}(s)|^2)g(|u^{\epsilon}(s)|^2)|u^{\epsilon}(s)|^2u^{\epsilon}(s)\>ds,\\
		M_{\alpha}(u^{\epsilon}(t))&=M_{\alpha}(u^{\epsilon}_0)
		+\int_0^t 4\alpha \<(1+|x|^2)^{\alpha-1}x u^{\epsilon}(s),\bi \nabla u^{\epsilon}(s)\>ds\\
		&\quad-2\int_0^t \<(1+|x|^2)^{\alpha}u^{\epsilon}(s), \bi \sum_{k\in \N^+}Im(Q^{\frac 12}e_i) Q^{\frac 12}e_i g'(|u^{\epsilon}(s)|^2)g(|u^{\epsilon}(s)|^2)|u^{\epsilon}(s)|^2u^{\epsilon}(s) \>ds\\
		&\quad +2\int_0^t\<(1+|x|^2)^{\alpha}u^{\epsilon}(s),\mathbf i g(|u^{\epsilon}(s)|^2)u^{\epsilon}(s)dW(s)\>.
		\end{align*}
	}
\end{prop}

\begin{proof}[Sketch proof of Lemma \ref{lm-con}]
	The H\"older regularity estimate is a consequence of \cite[Corollary 3.2]{C2020}.
	The proof of the strong convergence is similar to that of \cite[Theorem 1.1]{C2020}. By using Proposition \ref{prop-evo} and assumptions on $g$ and $f_{\epsilon},$ it is not hard to establish the following a priori estimate
	$$\E[\sup_{t\in [0,T]}\|u^{\epsilon}(t)\|_{\mathbb H^1}^p]+\E[\sup_{t\in [0,T]}\|u^{\epsilon}(t)\|_{L^2_{\alpha}}^p]\le C(Q,T,\lambda, p,u_0,\alpha).$$
	Denote $f_{0}(x)=\log(x).$ Then by applying It\^o formula to $\|u^{\epsilon}-u\|^2$, using integration by parts and the properties in Assumption \ref{main-as} and \ref{main-reg-fun}, we obtain the following error estimates.
	In the case of $\widetilde g=1$,  for $\eta'(d-2)\le 2,$ it holds that
	\begin{align*}
	&\quad\|u(t)-u^{\epsilon_n}(t)\|^2\\\nonumber
	&=\int_0^t 2\<u-u^{\epsilon_n},\bi \lambda f_{0}(|u|^2)u-f_{\epsilon_n}(|u^{\epsilon_n}|^2)u^{\epsilon_n}\>ds\\\nonumber
	&\le \int_0^t  4|\lambda| \|u(s)-u^{\epsilon_n}(s)\|^2 ds
	+4|\lambda| \int_0^t |Im \<u(s)-u^{\epsilon_n}(s),(f_{0}(|u|^2)-f_{\epsilon_n}(|u|^2)u\>| ds\\\nonumber
	&\le \int_0^t  6|\lambda| \|u^{\epsilon_m}(s)-u^{\epsilon_n}(s)\|^2 ds+4|\lambda| \int_0^t \|u^{\epsilon_m}(s)-u^{\epsilon_n}(s)\|_{L^1}
	\Big\|\frac {(\epsilon_m-\epsilon_n)|u^{\epsilon_n}|}{\epsilon_m+|u^{\epsilon_n}|^2}\Big\|_{L^{\infty}} ds\\\nonumber
	&\quad+2|\lambda|\int_0^t \|\log(1+\frac {(\epsilon_n-\epsilon_m)|u^{\epsilon_n}|^2}{1+\epsilon_m|u^{\epsilon_n}|^2}) |u^{\epsilon_n}|\|^2ds\\\nonumber 
	&\le \int_0^t  4|\lambda| \|u^{\epsilon_m}(s)-u^{\epsilon_n}(s)\|^2 ds
	+4|\lambda|\epsilon_n^{\frac 12} \int_0^t \|u^{\epsilon_m}(s)-u^{\epsilon_n}(s)\|_{L^1}ds\\\nonumber
	&\quad +2|\lambda| C\epsilon_n^{\eta'}\int_0^t \|u^{\epsilon_n}\|_{L^{2+2\eta'}}^{2+2\eta'}ds.
	\end{align*}
	In the case of $\widetilde g(x)=\mathbf i g(|x|^2)x, $ for $\eta'(d-2)\le 2,$
	\begin{align*}
	&\|u^{\epsilon_m}(t)-u^{\epsilon_n}(t)\|^2\\\nonumber
	&\le \int_0^t  (4|\lambda|+C(g,Q)) \|u^{\epsilon_m}(s)-u^{\epsilon_n}(s)\|^2 ds
	+4|\lambda|\epsilon_n^{\frac 12} \int_0^t \|u^{\epsilon_m}(s)-u^{\epsilon_n}(s)\|_{L^1}ds\\\nonumber
	&+2|\lambda| C\epsilon_n^{\eta'}\int_0^t \|u^{\epsilon_n}\|_{L_{2+2\eta'}}^{2+2\eta'}ds+\int_0^t\<u^{\epsilon_m}-u^{\epsilon_n}, \mathbf i \Big(g(|u^{\epsilon_m}|^2)u^{\epsilon_m}-g(|u^{\epsilon_n}|^2)u^{\epsilon_n}\Big)dW(s)\>.
	\end{align*}
	Combining the above error estimates with the a priori estimates of $u^{\epsilon}$ and following the steps in the proof of \cite[Theorem 1.1]{C2020}, we obtain the desired convergence rate.
\end{proof}

\bibliographystyle{plain}
\bibliography{bib}

\begin{thebibliography}{10}

\bibitem{AC18}
R.~Anton and D.~Cohen.
\newblock Exponential integrators for stochastic {S}chr\"{o}dinger equations
  driven by {I}t\^{o} noise.
\newblock {\em J. Comput. Math.}, 36(2):276--309, 2018.

\bibitem{AZ03}
A.~V. Avdeenkov and K.~G. Zloshchastiev.
\newblock Quantum bose liquids with logarithmic nonlinearity:
  Self-sustainability and emergence of spatial extent.
\newblock {\em J. Phys. B}, 44:195--303, 2011.

\bibitem{BRZ17}
V.~Barbu, M.~R\"{o}ckner, and D.~Zhang.
\newblock The stochastic logarithmic {S}chr\"{o}dinger equation.
\newblock {\em J. Math. Pures Appl. (9)}, 107(2):123--149, 2017.

\bibitem{BM76}
I.~Bialynicki-Birula and J.~Mycielski.
\newblock Nonlinear wave mechanics.
\newblock {\em Ann. Physics}, 100(1-2):62--93, 1976.

\bibitem{BC2020a}
C.~E. Br\'ehier and D.~Cohen.
\newblock Analysis of a splitting scheme for a class of nonlinear stochastic
  {S}chr\"odinger equations.
\newblock {\em arXiv:2007.02354}.

\bibitem{CG18}
R.~Carles and I.~Gallagher.
\newblock Universal dynamics for the defocusing logarithmic {S}chr\"{o}dinger
  equation.
\newblock {\em Duke Math. J.}, 167(9):1761--1801, 2018.

\bibitem{Caz83}
T.~Cazenave.
\newblock Stable solutions of the logarithmic {S}chr\"{o}dinger equation.
\newblock {\em Nonlinear Anal.}, 7(10):1127--1140, 1983.

\bibitem{CH16}
C.~Chen and J.~Hong.
\newblock Symplectic {R}unge--{K}utta {S}emidiscretization for {S}tochastic
  {S}chr\"odinger {E}quation.
\newblock {\em SIAM J. Numer. Anal.}, 54(4):2569--2593, 2016.

\bibitem{CH17}
J.~Cui and J.~Hong.
\newblock Analysis of a splitting scheme for damped stochastic nonlinear
  {S}chr\"{o}dinger equation with multiplicative noise.
\newblock {\em SIAM J. Numer. Anal.}, 56(4):2045--2069, 2018.

\bibitem{CHL16b}
J.~Cui, J.~Hong, and Z.~Liu.
\newblock Strong convergence rate of finite difference approximations for
  stochastic cubic {S}chr\"odinger equations.
\newblock {\em J. Differential Equations}, 263(7):3687--3713, 2017.

\bibitem{CHLZ17}
J.~Cui, J.~Hong, Z.~Liu, and W.~Zhou.
\newblock Strong convergence rate of splitting schemes for stochastic nonlinear
  {S}chr\"{o}dinger equations.
\newblock {\em J. Differential Equations}, 266(9):5625--5663, 2019.

\bibitem{C2020}
J.~Cui and L.~Sun.
\newblock Stochastic logarithmic {S}chr\"odinger equations: energy regularized
  approach.
\newblock {\em arXiv:2102.12607}.

\bibitem{BD06}
A.~de~Bouard and A.~Debussche.
\newblock Weak and strong order of convergence of a semidiscrete scheme for the
  stochastic nonlinear {S}chr\"odinger equation.
\newblock {\em Appl. Math. Optim.}, 54(3):369--399, 2006.

\bibitem{Hef85}
E.~F. Hefter.
\newblock Application of the nonlinear {S}chr\"odinger equation with a
  logarithmic inhomogeneous term to nuclear physics.
\newblock {\em Phys. Rev. A}, 32(3):1201--1204, 1985.

\bibitem{HW19}
J.~Hong and X.~Wang.
\newblock {\em Invariant measures for stochastic nonlinear {S}chr\"{o}dinger
  equations}, volume 2251 of {\em Lecture Notes in Mathematics}.
\newblock Springer, Singapore, 2019.
\newblock Numerical approximations and symplectic structures.

\bibitem{MFGL03}
S.~D. Martino, M.~Falanga, C.~Godano, and G.~Lauro.
\newblock Logarithmic {S}chr\"odinger-like equation as a model for magma
  transport.
\newblock {\em Europhys. Lett.}, 63:472--475, 2003.

\bibitem{Yas78}
K.~Yasue.
\newblock Quantum mechanics of nonconservative systems.
\newblock {\em Ann. Phys.}, 114:479--496, 1978.

\bibitem{Zlo10}
K.~G. Zloshchastiev.
\newblock Logarithmic nonlinearity in theories of quantum gravity: origin of
  time and observational consequences.
\newblock {\em Gravit. Cosmol.}, 16(4):288--297, 2010.

\end{thebibliography}

\end{document}